\documentclass{article}
\usepackage[utf8]{inputenc}
\usepackage[left=2cm,right=2cm,top=2cm,bottom=2cm]{geometry}
\usepackage{lineno,hyperref}
\modulolinenumbers[5]
\usepackage{amsmath}
\usepackage{amssymb}
\usepackage{float}
\usepackage{tikz}
\usepackage{mathtools}
\usepackage{hyperref}
\usepackage{lipsum}
\usepackage{caption}
\usepackage{subcaption}
\usepackage{commath}
\usepackage{cases}
\usepackage{amsthm}
\usepackage{graphicx}
\usepackage{listings}
\usepackage{multicol}
\usepackage{fancyhdr}
\usepackage{algorithm}
\usepackage{algpseudocode}

\usepackage[sc,osf]{mathpazo}
 \newcommand{\spann}{\textrm{span}}

   \newcommand{\wHat}{\widehat}

\newcommand{\uu}{u}

\newcommand{\xx}{\textbf{x}}

\newcommand{\vv}{v}

\newcommand{\wt}{\widetilde{t}}

\newcommand{\pt}{\partial}

\newcommand{\sens}{sensitivity }
\newcommand{\sensfct}{sensitivities}

\definecolor{darkWhite}{rgb}{0.94,0.94,0.94}

\newtheorem{theorem}{Theorem}[section]

\newtheorem{corollary}[theorem]{Corollary}

\newtheorem{remark}[theorem]{Remark}
 
\begin{document}

 \title{\noindent The non-intrusive reduced basis two-grid method applied to sensitivity analysis}

\maketitle
\normalsize
\begin{center}
\author{Elise Grosjean \footnotemark[1],}
\author{Bernd Simeon \footnotemark[1]}
\end{center}

\footnotetext[1]{Department of Mathematics, RPTU Kaiserslautern-Landau, 67657, Deutschland}
\date

\begin{abstract}
  This paper deals with the derivation of Non-Intrusive Reduced Basis (NIRB) techniques for sensitivity analysis, more specifically the direct and adjoint state methods.
  For highly complex parametric problems, these two approaches may become too costly ans {}{thus Reduced Basis Methods (RBMs) may be a viable option.} We propose new NIRB two-grid algorithms for both the direct and adjoint state methods {}{in the context of parabolic equations}. The NIRB two-grid method uses the HF code solely as a ``-box'', requiring no code modification. Like other RBMs, it is based on an offline-online decomposition. The offline stage is time-consuming, but it is only executed once, whereas the online stage {}{employs coarser grids and thus,} is significantly less expensive than a {}{fine} HF evaluation. \\
  On the direct method, we prove on a classical model problem, the heat equation, that HF evaluations of sensitivities reach an optimal convergence rate in $L^{\infty}(0,T;H^1_0(\Omega))$, and then establish that these rates are recovered by the NIRB two-grid approximation. These results are supported by numerical simulations. We then propose a {}{new procedure that} further reduces the computational costs of the online step while only computing a coarse solution of the state equations. {}{On the adjoint state method, we propose a new algorithm that reduces both the state and adjoint solutions.} All numerical results are run with the model problem as well as a more complex problem, namely the Brusselator system.
\end{abstract}


\section{Introduction.}
\label{introduction}

Sensitivity analysis is a critical step in optimizing the parameters of a parametric model.
The goal is to see how sensitive its results are to small changes of its input parameters.

Several methods have been developed for computing sensitivities, see \cite{reviewSA} for an overview. 
We focus here on two differential-based sensitivity analysis approaches, in connection with models given 
as reaction-diffusion equations.
\begin{itemize} 
\item {}{The first method which we considered is called ``the direct method'', and is also known as the "forward method".} It may be used when dealing with discretized solutions of parametric Partial Differential Equations (PDEs). The sensitivities (of the solution or other outputs of interest) are computed directly from the original problem. One drawback is that it necessitates solving a new system for each parameter of interest, i.e., for $P$ parameters of interest, $P+1$ problems have to be solved.
\item {}{The second method we considered is ``the adjoint state method'', also known as the "backward method".} It may be a viable option \cite{sykes1984adjoint} when the direct method becomes prohibitively expensive. In this setting, the goal is to compute the sensitivities of an objective function that one aims at minimizing. The associated Lagrangian is formulated, and by choosing appropriate multipliers, a new system known as "the adjoint" is derived. This approach is preferred in many situations since it avoids calculating the sensitivities with respect to the solutions. For example, in the framework of inverse problems, one can determine the "true" parameter from several {}{measurements} (which are usually provided by multiple sensors) while combining it with a gradient-type optimization algorithm. As a result, one gets the "integrated effects" on the outputs over a time interval. The advantage is that it only requires two systems to solve regardless of the number of parameters of interest. 
\end{itemize}
{}{Thus, the direct method is appealing when there are relatively few parameters, whereas the adjoint state method is preferred when there are many parameters.}

\paragraph*{Earlier works.}
For extremely complex simulations, both methods may still be impractical {}{with a High-Fidelity (HF) classical solver}. Several reduction techniques have thus been investigated in order to reduce the complexity of the sensitivity computation. Among them, Reduced Basis Methods (RBMs) are a well-developed field \cite{maday2006reduced,peterson1989reduced,binev2011convergence}. They use an offline-online decomposition, in which the offline step is time-consuming but is only performed once, and the online step is significantly less expensive than a fine HF evaluation. In the context of sensitivity analysis, the majority of these studies relies on a {}{RB-}Galerkin projection \cite{rb} onto the adjoint state system in the online part. In what follows, we present a brief review of previous works on RBMs combined with both sensitivity methods. 
\begin{itemize}
\item Let us begin with the direct method. It has been employed and studied with Reduced Basis (RB) spaces in various applications, e. g., \cite{noor1993reduced,watson1996sensitivity,ding2021design}.
In \cite{KIRSCH1994143}, {}{to improve the precision of the sensitivity approximations}, a combined method is proposed (based on local and global approximations with series expansion and a RB expression), which was first developed in \cite{kirsch1991reduced}. {}{Conversely, the sensitivities may be used to enhance the RB generation \cite{hay2008use}, as in \cite{hay_borggaard_pelletier_2009}, where the authors examine two strategies that use the sensitivity of the Proper Orthogonal Decomposition (POD) modes with respect to the problem parameters to improve the flows representation in the context of Navier-Stokes equations}.
      
  \item  Concerning the adjoint state formulation, its first applications in conjunction with computational reduction approaches can be found in \cite{ito1998reduced} in the context of RBMs, where several RB sub-spaces are compared, or in \cite{ravindran2000reduced} with the POD method, where the underlying PDE has an affine
parameter dependence. Particular emphasis is being placed on developing accurate a-posteriori error estimates in order to improve the reduced basis generation \cite{quarteroni2007reduced,tonn2011comparison,dede2010reduced,dede2012reduced}. {}{In recent studies of optimal control problems, RBMs have been simultaneously employed on the state, adjoint and control variables \cite{karcher2018certified,bader2016certified}.}
  
  Even if the adjoint state method is frequently preferred, writing its associated reduced problem can be difficult when the adjoint formulation is not straightforward. It may also be reformulated to take advantage of previously developed RB theory.
For example, in \cite{negri2013reduced}, it is rewritten as a saddle-point problem for Stokes-type problems. 

\end{itemize}
We would like to note also that variance-based sensitivity analysis {}{(with Sobol indices)} has been investigated using RBMs \cite{janon2016goal} and in particular Non-Intrusive Reduced Basis (NIRB) methods  \cite{kumar2016efficient}.\\ 
{}{
In this paper, we focus on the NIRB two-grid method, which has been developed to approximate state direct problems in the context of elliptic equations \cite{madaychakir} with the Finite Element Method (FEM) and has been generated to finite volume schemes in \cite{VF}. Recently, it has been extended to parabolic equations \cite{paraboliccontext}. To conclude this brief overview of RBMs applied to sensitivity analysis, note that the present work adapts the NIRB two-grid method to sensitivity analysis thanks to a prediction of the RB coefficients, an idea that has already been tested to some extent with Radial Basis Functions (RBF) or neural networks approaches \cite{hestaven2018non,haasdonk2023new}.} 
\paragraph*{Motivation.}

Even though the {}{RB-Galerkin} method is prevalent in the literature, {}{in general,} its main disadvantage lies in its intrusiveness. {}{Indeed, if the PDE has a non-linear dependency with respect to the parameter \cite{barrault2004empirical} or if the whole matrix (or the right-hand side) of the underlying problem cannot be retrieved, the assembly routines computed from its formulation must be changed in the HF code, leading to an intrusive procedure \cite{rb}. This may be difficult if the HF code is very complex or even impossible if it has been purchased, as is often the case in an industrial context. 
 Non-Intrusive Reduced Basis (NIRB) methods are more practical to implement than other RBMs because they require the execution of the HF code as a "-box" solver only. Several nuances of non-intrusivity coexist, with the goal of making RBMs adaptable to a wider range of softwares. The methods we shall consider are non-intrusive in the sense that they require the knowledge of the spatial and temporal coordinates, solutions of the state equations and of the underlying sensitivity problems only. Note that if the mesh connectivity is not known, the method still works, as shown in the context of elliptic equations in \cite{grosj2022}. It should be noted that some solvers can provide sensitivity solutions via algorithm differentiation (see e.g. SU2 solver \cite{economon2016su2}).}
{}{
Apparently, NIRB methods have not yet been used to approximate sensitivities except for statistical approaches such as variance-based sensitivity analysis \cite{janon2016goal}.
}

{}{
In the present study, we shall focus on the NIRB two-grid method \cite{madaychakir,VF,inproceedings,censity} (see also other NIRB methods \cite{Casenave2014,audouze2013nonintrusive,willcox}, that differ from the two-grid method). Like most RBMs, the NIRB two-grid method relies on the assumption that the manifold of all solutions has a small Kolmogorov width~\cite{kolmo}. Its main advantages are that it can be used for a wide range of PDEs and that it is simple to implement. Furthermore, its non-intrusivity makes it appealing from an engineering standpoint, as explained above.
The effectiveness of this method depends on its offline/online decomposition (as for most RBMs). The offline part, which generates the RB, takes time but is performed once only.  The specific feature of this NIRB approach is to solve the parametric problem {}{with the solver} solely on a coarse mesh during the online step, and then to rapidly improve the precision of the coarse solution by projecting it onto the reduced space. It makes this portion of the algorithm much cheaper than a fine HF evaluation. 
}

\paragraph*{Main results.}

{}{
Our main results may be summarized as follows:
\begin{itemize}    \item \textbf{Direct method}. We have carried out a thorough theoretical analysis of the heat equation as a model problem. In this setting, we derive optimal convergence rates in $L^{\infty}(0,T;H^1_0(\Omega))$ for the fully-discretized sensitivity solution. It turns out that we obtain these optimal rates {}{theoretically and numerically} also for the NIRB sensitivity approximations. Our main theoretical result is given by Theorem \ref{EllipticEst}. Then, drawing inspiration from recent works \cite{grosjean:hal-03588508,GUO2018807}, we efficiently apply a supervised learning process that makes the online stage much quicker. Indeed, only the state equations have to be solved, on the coarse mesh. In particular, the online stage no longer requires to solve as many systems as there are parameters, hence making the direct method more competitive with the adjoint one.
\item \textbf{Backward method}. In this context, we adapt the NIRB two-grid algorithm by applying two reductions (both on the state and on the ajoint solutions) to the underlying problem. Finally, we numerically demonstrate its efficacy on the heat equation and on the Brusselator problem.
\end{itemize}
}


\paragraph*{Outline  of  the  paper.}

The  remainder of this paper  is  organized  as  follows. Section \ref{mathbackground} describes both sensitivity methods along with established convergence results and the NIRB two-grid algorithm for parabolic equations. {}{Section \ref{NIRBproof}} is devoted to the theoretical results on the rate of convergence for the NIRB sensitivity approximation. {}{In Section \ref{NIRBalgoS}, we present the new NIRB algorithms for the direct and adjoint methods.} In the last Section \ref{results}, {}{several numerical results are presented to support the theoretical results.}\\


\section{Mathematical Background.}

\label{mathbackground}
Let $\Omega$ be a bounded domain in $\mathbb{R}^d$ (with $d \in \mathbb{N}^+$), and a {}{sufficiently smooth} boundary $\pt \Omega$, and consider a parametric problem $\mathcal{P}$ on $\Omega$. Let $\mu \in \mathcal{G}$ be its parameter of interest. {}{For each parameter value $\mu \in \mathcal{G}$, we denote $u(\mu) \in V$ the associated solution of $\mathcal{P}$, where $V$ is a suitable Banach space. In what follows, we consider homogeneous Dirichlet boundary conditions, so that $V := H_0^1(\Omega)$, with the associated norm $|\cdot|_{H^1(\Omega)}$.} Note that we will use the notation $(\cdot,\cdot)$ to denote the classical $L^2(\Omega)$-inner product.

In this section, we shall introduce our model problem, that of the heat equation, in a continuous setting, and then its spatial and time discretizations. Then, we detail the sensitivity problems associated to the heat equation and recall the NIRB algorithm in the context of parabolic equations.  \\

In the next sections, $C$ will denote various positive constants independent of the size of the spatial and temporal grids and of the parameter $\mu$, and $C(\mu)$ will denote constants independent of the sizes of the grids but dependent on $\mu$. 
\subsection{A model problem: The heat equation.}
\subsubsection{The continuous problem.}

{}{Our aim is to find accurate sensitivity approximations of the heat equation.} We thus first consider the following heat equation on the domain $\Omega$ with homogeneous Dirichlet {}{boundary} conditions, which takes the form

\begin{numcases}{}
& $u_t-\nabla \cdot (A(\cdot; \mu) \nabla u) =f,\ \textrm{ in }\Omega\times ]0,T]$, \nonumber\\
& $u(\cdot,0;\mu)=u^0(\cdot; \mu), \ \textrm{ in }\Omega$,   \label{heatEq2} \\
& $u(\cdot,t)=0, \ \textrm{ on }\pt \Omega \times ]0,T]$, \notag
\end{numcases} 
where
 \begin{align}
\label{assumptionparabolic}
& {}{   \textrm { the parameter }\mu=(\mu_1,\cdots,\mu_P) \in \mathcal{G} \subset \mathbb{R}^P \textrm{ and } 
     \mathcal{C}^2(\mathcal{G}) \ni A:\Omega \times \mathcal{G} \to {}{\mathbb{R}^{d\times d}} \textrm{ is measurable, bounded, }} \notag \\
     & {}{\textrm{and uniformly elliptic}.  \textrm{ The function f }\in L^2(\Omega \times [0,T] ),\textrm{ while }u^0(\mu) \in H_0^1(\Omega) {}{\cap \mathcal{C}^1(\mathcal{G})}.}
    \end{align}
    
     {}{ We assume that under these assumptions, the state solution is differentiable with respect to $\mu$ \cite{chavent1974identification} (see \cite{cacuci1980sensitivity} for general nonlinear equations).}
    For any $t>0$, the solution {}{is} $u(\cdot,t) \in H_0^1(\Omega)$, and $u_t(\cdot, t) \in L^2(\Omega)$ stands for the derivative of $u$ with respect to time.
   
We use the conventional notations for space-time dependent Sobolev spaces \cite{lions1961problemes}
\begin{align*}
 & L^p(0,T;V):=\Big\{u\ | \ \norm{u}_{L^p(0,T;V)}:=\Big(\int_0^T \norm{u(\cdot,t)}_{V}^p \ dt \Big)^{1/p}< \infty \Big\}, \ 1\leq  p < \infty,\\
  &    L^{\infty}(0,T;V):=\Big\{u\ | \ \norm{u}_{L^{\infty}(0,T;V)}:= \ \underset{0\leq t \leq T }{\mathrm{ess\ sup}} \ \norm{u(\cdot, t)}_V< \infty \Big\},
\end{align*}
We shall employ FEM for the spatial discretization, and thus consider the variational formulation of \eqref{heatEq2}:\\
\begin{numcases}{}
  & Find $u\in L^2(0,T;H_0^1(\Omega))$ with $u_t \in L^2(0,T;H^{-1}(\Omega))$ such that \nonumber \\
 & $ (u_t(\cdot,t),v)+a(u(\cdot,t),v;\mu)=(f(\cdot,t),v), \ \forall v \in H_0^1(\Omega) \textrm{ and } t\in (0,T)$,  \label{varpara} \\
 & $ u(\cdot, 0;\mu)=u^0(\cdot;\mu) , \textrm{ in } \Omega,$ \nonumber
\end{numcases}
where $a$ is given by \begin{equation}
  a(w,v;\mu)=\int_{\Omega}A(\xx;\mu) \nabla w(\xx)\cdot \nabla v (\xx)\ d\xx,\quad \forall w, v\in H_0^1(\Omega).
  \label{ateerm1}
\end{equation}
We remind that \eqref{varpara} is well posed (see \cite{evans10} for the existence and the uniqueness of solutions to problem \eqref{varpara}) and we refer to the notations of \cite{evans10}. 

\subsubsection{The various discretizations.}
To approximate our solution via the NIRB algorithm, we shall consider two spatial grids of $\Omega$:
\begin{itemize}
  \item one fine mesh, denoted $\mathcal{T}_h$,  where its size $h$ is defined
as \begin{equation}
  \label{finemeshsize}
  h = \underset{K\in \mathcal{T}_h}{\textrm{max }} h_K,
\end{equation}
\item and on coarse mesh, denoted $\mathcal{T}_H$, with its size defined as
  \begin{equation}
    \label{coarsemeshsize}
    H = \underset{K\in \mathcal{T}_H}{\textrm{max }} H_K >> h,
    \end{equation}
\end{itemize}
where the diameter $h_K$ (or $H_K$) of any element $K$ in a mesh is equal to $\underset{x,y \in K}{\sup \ } |x-y|$.\\
We employed $\mathbb{P}_1$ finite elements to discretize in space. Thus, we introduce $V_h$ and $V_H$, the continuous piecewise linear finite element functions (on fine and coarse meshes, respectively) that vanish on the boundary $\pt \Omega$.
We consider the so-called Ritz projection operator  $P^1_h: H^1_0(\Omega) \to V_h$ ($P^1_H$ on $V_H$ is defined similarly) which is given by
\begin{equation}
(\nabla P^1_h u, \nabla v)=(\nabla u, \nabla v),\quad  \forall v \in V_h, \ \textrm{for } u \in H^1_0(\Omega).
  \label{P1opParabol}
\end{equation}
In the context of time-dependent problems, a time stepping method of finite difference type is used to get a fully discrete approximation of the solution of \eqref{heatEq2}. 
  {}{  In analogy with the spatial discretizations, we consider two different time grids for the time discretization:}
\begin{itemize}
\item One time grid, denoted $F$, is associated to fine solutions (for the generation of the snapshots). To avoid making notations more cumbersome, we will consider a uniform time step $\Delta t_F$. The time levels can be written $t^n = n \ \Delta t_F$, where $n \in {}{\mathbb{N}^+}$.
\item Another time grid, denoted $G$, is used for coarse solutions. By analogy with the fine grid, we consider a uniform grid with time step $\Delta t_G$. Now, the time levels are written $\widetilde{t}^m = m \ \Delta t_G$, where $m \in \mathbb{N}^+$. 
\end{itemize}

As in the elliptic context \cite{madaychakir}, the NIRB algorithm is designed to recover the optimal estimate in space. To do so, the NIRB theory uses {}{the} Aubin-Nitsche lemma in the context of FEM. Yet, since there is no such argument as the Aubin-Nitsche lemma for time stepping methods, we must consider time discretizations that provide the same precision with larger time steps, as detailed in \cite{paraboliccontext}. Thus, we consider a higher order time scheme for the coarse solution.
As in \cite{paraboliccontext}, we used an Euler scheme (first order approximation) for the fine solution and the Crank-Nicolson scheme (second order approximation) for the coarse solution on our model problem.\\
Thus, we deal with two kinds of notations for the discretized solutions:
\begin{itemize}
\item $u_h(\xx,  t)$ and $u_H(\xx,t)$ (or $u_h(\xx,  t;\mu)$ and $u_H(\xx,t;\mu)$ in order to highlight the $\mu$-dependency) that respectively denote the fine and coarse solutions of the spatially semi-discrete solution, at time $t \geq 0$.
\item $u_h^n(\xx)$ and $u_H^m(\xx)$ (or $u_h^n(\xx;\mu)$ and $u_H^m(\xx;\mu)$) that respectively denote the fine and coarse full-discretized solutions at time $t^n=n \times \Delta t_F$ and $\widetilde{t}^m = m\times \Delta t_G$.
\end{itemize}

\begin{remark}
  To simplify notations, we consider that both time grids end at time $T$ here,
  \begin{equation*}
    T\ =\ N_T\ \Delta t_F \ =\  M_T\ \Delta t_G.
  \end{equation*}     
\end{remark}

The semi-discrete form of the variational problem \eqref{varpara} {}{reads} for the fine mesh (similarly for the coarse mesh):
\begin{numcases}{}
  & Find $u_h(t)=u_h(\cdot,t) \in V_h$ for $t\in [0,T]$ such that \nonumber \\
 & $ (u_{h,t}(t),v_h)+a(u_h(t),v_h;\mu)=(f(t),v_h), \ \forall v_h \in V_h \textrm{ and } t \in ]0,T]$, \label{varpara1disc} \\
 & $ u_h(\cdot, 0;\mu)\ =\ u_h^0(\cdot;\mu) \ =\ P_h^1(u^0)(\cdot;\mu)$. \nonumber
\end{numcases}
From the definition of the Ritz projection operator $P_h^1$ \eqref{P1opParabol}, the initial condition $u_h^0(\mu)$ (and similarly for the coarse mesh) is such that
\begin{equation}
  (\nabla u_h^0(\mu),\nabla v_h)=(\nabla u^0(\mu),\nabla v_h), \forall v_h \in V_h,
\end{equation}
and hence, it corresponds to the finite element approximation of the corresponding elliptic problem whose exact solution is $u^0(\mu)$.\\

The full discrete form of the variational problem \eqref{varpara} for the fine mesh with an implicit Euler scheme {}{can be expressed as}:
\begin{numcases}{}
  &  Find $u_h^n \in V_h$ for $n= 0,\dots,N_T$ such that \nonumber \\
 & $ (\overline{\pt} u_h^n,v_h)+a(u_h^n,v_h;\mu)=(f(t^n),v_h), \ \forall v_h \in V_h \textrm{ and } n = 1,\dots,N_T$,\label{varpara2disc}\\
 & $ u_h(\cdot, 0;\mu)\ =\ u_h^0 (\cdot;\mu)$, \nonumber
\end{numcases}
where the time derivative in the variational form of the problem \eqref{varpara1disc} has been replaced by a backward difference quotient, $\overline{\pt} u^n_h=\frac{u^n_h-u^{n-1}_h}{\Delta t_F}$.\\
For the coarse mesh with the Crank-Nicolson scheme, and with the notation $\overline{\pt} u^m_H=\frac{u^m_H-u^{m-1}_H}{\Delta t_G}$, it becomes:
\begin{numcases}{}
  &  Find $u_H^m \in V_H$ for $m = 0,\dots,M_T$, such that \nonumber \\
 & $ (\overline{\pt} u_H^m,v_H)+a(\frac{u_H^m + u_H^{m-1}}{2},v_H;\mu)=(f(\widetilde{t}^{m-\frac{1}{2}}),v_H), \ \forall v_H \in V_H \textrm{ and } m = 1,\dots M_T$, \nonumber \\
 & $ u_H(\cdot, 0;\mu)\ =\ u_H^0 (\cdot;\mu)$, \label{varpara2discCN}
\end{numcases}
where $\widetilde{t}^{m-\frac{1}{2}}= \frac{\widetilde{t}^{m}+ \widetilde{t}^{m-1}}{2}.$

To approximation the solution $u(\mu)$ with the NIRB two-grid method, as explained in \cite{paraboliccontext}, we will need to interpolate in space (as for elliptic equations) and in time the coarse solution. So let us introduce the quadratic interpolation in time of a coarse solution at time $t^n \in I_m=[\widetilde{t}^{m-1},\widetilde{t}^m]$ defined on $[\widetilde{t}^{m-2},\widetilde{t}^m]$ from the coarse approximations at times $\wt^{m-2}, \wt^{m-1},$ and  $\wt^m$,\ for all $m=2,\dots, M_T$.
To this purpose, we employ the following parabola on $[\wt^{m-2},\wt^m]$: \\
    For $m\geq 2$, $\forall {}{t^n} \in I_m=[\widetilde{t}^{m-1},\widetilde{t}^m]$,
    \begin{align}
      \label{parabola}
     {}{ I^2_n[u_H^m](\mu):=}&{}{u_H^{m-2}(\mu)\frac{(\wt^m-t^n)(t^n-\wt^{m-1})}{(\widetilde{t}^m - \wt^{m-2})(\wt^{m-2}-\wt^{m-1})}  + u_H^{m-1}(\mu)\frac{ (\wt^{m-2}-t^n)(t^n-\wt^{m}) }{(\widetilde{t}^{m-2} - \wt^{m-1})(\wt^{m-1}-\wt^{m})}} \nonumber\\
      &\quad \quad  {}{+ u_H^{m}(\mu)\frac{ (\wt^{m-1}-t^n)(t^n-\wt^{m-2})}{(\widetilde{t}^{m-1} - \wt^{m})(\wt^{m}-\wt^{m-2})}.}
     \end{align}
       For $t^n \in I_1=[\wt^0,\wt^1]$, we use the same parabola defined by the coarse approximations at times $\wt^0, \ \wt^1, \ \wt^2$ as the one used over $[\wt^1,\wt^2]$.
  We denote  by $\widetilde{u_H}^n(\mu)=I^2_n[u_H^m](\mu)$ the {}{piecewise} quadratic interpolation of $u_H^m$ at a time $t^n$.
  Note that we choose this interpolation in order to keep an approximation of order 2 in time (it works also with other quadratic interpolations).

\subsection{Sensitivity analysis: The direct problem.}
In this section, we shall recall the sensitivity systems (continuous and discretized versions) for $P$ parameters of interest. Then, with the direct formulation, we prove several numerical results on the model problem which are used for the proof of the NIRB error estimate. To not make the notations too cumbersome in the proofs, we consider $A(\mu)=\mu \ I_d$, with $P=1$ and $\mu \in \mathbb{R}^+_* ${}{$(=\mathbb{R}_{>0})$} for the analysis and we drop the $\mu$-dependency notations for $u$ and $u_0$. 
\subsubsection{The continuous setting. }
We consider $P$ parameters of interest, denoted $\mu_p=1,\dots,P$. We aim at approximating the exact derivatives, also known as the \textit{sensitivities}
\begin{equation}
  \label{sensitivities}
\boxed{  \Psi_p(t,\xx;\mu):=\frac{\pt u}{\pt \mu_p}(t,\xx;\mu), \quad \textrm{in } [0,T]\times \Omega.}
\end{equation}
To do so via the direct method, we solve $P$ new systems, which can be directly obtained by differentiating the state equations with respect to each $\mu_p$.
The continuous state problem \eqref{varpara} may be rewritten
\begin{numcases}{}
  &  Find $u(t) \in V$ for $t \in [0,T]$ such that \nonumber \\
  & $ (u_t(t),v)=F(u(t),v;\mu):=-a(u(t),v;\mu)+(f(t),v), \ \forall v \in V,\   t > 0$, \notag \\
 & $ u(\cdot, 0;\mu)\ =\ u^0(\cdot;\mu) $, \notag \label{varparasensheat}
\end{numcases}
where the bilinear form $a$ is defined by \eqref{ateerm1}.
Using the chain rule and since the time and the parameter derivatives commute,
\begin{equation*}
  (\Psi_{p,t}(t),v)= \frac{\pt F}{\pt u}(u(t),v;\mu) \cdot \Psi_p (t) + \frac{\pt F}{\pt \mu}(u(t),v).
\end{equation*}
 We therefore obtain the following problem
\begin{numcases}{}
  \label{eq:2chainrule}
  & Find $\Psi_p(t) \in V$ for $t \in [0,T]$ such that \notag \\
&$( \Psi_{p,t}(t),v)+a\big(\Psi_p(t),v;\mu\big)= - \big( \frac{\pt A}{\pt \mu_p}(\mu)\nabla u(t),\nabla v\big), \textrm{ for } v \in V, \textrm{ for } t >0,$ \notag \\
& $\Psi^0_p(\mu) = \frac{\pt u^0}{\pt \mu_p}(\xx;\mu),$
\end{numcases}
which is well-posed since $u \in L^2(0,T;\ H_0^1(\Omega))$, and under the assumptions \eqref{assumptionparabolic}, the so-called "parabolic regularity estimate" implies that  $u \in L^2(0,T;\ H^2(\Omega)) \cap L^{\infty}(0,T;\ H_0^1(\Omega))$ \cite{evans10,thomee2}.
\subsubsection{The spatially semi-discretized version.}
To derive the NIRB approximation, as previously for the state solution, we discretize the \sens  problems \eqref{eq:2chainrule} in space and in time.\\
The corresponding spatially semi-discretized formulation on $\mathcal{T}_h$ (similarly on $\mathcal{T}_H$) reads
\begin{numcases}{}
  \label{semidiscretizedpsi}
  &  Find $\Psi_{p,h}(t) \in V_h$ for $t \in [0,\dots, T]$ such that \nonumber \\
  &$(\Psi_{p,h,t}(t),v_h)+a\big(\Psi_{p,h}(t),v_h;\mu\big)= - \big( \frac{\pt A}{\pt \mu_p} (\mu) \nabla u_{h}(t;\mu), \nabla v_h\big),\textrm{ for } v_h \in V_h,  \textrm{ for } t \in ]0,T]$, \notag \\
& $\Psi_{p,h}^0(\cdot;\mu) = P^1_h(\Psi^0_p)(\cdot;\mu),$
\end{numcases}
where $P_h^1$ is the Ritz projection operator given by \eqref{P1opParabol}.

\subsubsection{The fully-discretized versions.}
From \eqref{semidiscretizedpsi}, we can derive the fully-discretized systems for the fine and coarse grids.\\
The direct sensitivity problems with respect to the parameter $\mu_p$ on the fine mesh $\mathcal{T}_h$ with an Euler scheme read
\begin{numcases}{}
  \label{eq:directheatsensitivity}
  &  Find $\Psi_{p,h}^n \in V_h$ for $n \in \{0,\dots, N_T\}$ such that \nonumber \\
  &$(\overline{\pt} \Psi_{p,h}^n,v_h)+a\big(\Psi_{p,h}^n,v_h;\mu\big)= - \big(\frac{\pt A}{\pt \mu_p}(\mu)\nabla u_{h}^n(\mu), \nabla v_h\big) , \textrm{ for } v_h \in V_h, \textrm{ for } n=\{1,\dots,N_T\}$, \\
& $\Psi_{p,h}^0(\cdot;\mu) = P^1_h \Psi_p^0(\cdot;\mu),$ \notag
\end{numcases}
where, as before, the time derivative in the variational form of the problem \eqref{varparasensheat} has been replaced by a backward difference quotient, $\overline{\pt} \Psi^n_h=\frac{\Psi^n_h-\Psi^{n-1}_h}{\Delta t_F}$.\\
Before proceeding with the proof of Theorem \eqref{EllipticEst}, we need several theoretical results that can be deduced from \cite{thomee2}, but require some precisions. Indeed, first, in \cite{thomee2}, the estimates are proven on the heat equation with a non-varying diffusion coefficient. Secondly, the right-hand side function $f$ vanishes when seeking the error estimates, whereas in our case, the right-hand side function depends on $u$ and necessitates {}{more precise} estimates.\\

With the fully-discretized version \eqref{eq:directheatsensitivity}, the following estimate holds.
\begin{theorem}
  \label{corollaryPsii}
  Let $\Omega$ be a convex polyhedron. Let $A(\mu)=\mu \ I_d$, with $\mu \in \mathbb{R}_*^+$. \\
  Consider $u \in H^{1}(0,T;\ H^2(\Omega)) \cap H^{2}(0,T;\ L^2(\Omega))$ {}{to} be the solution of \eqref{heatEq2} with $u^0 \in H^2(\Omega)$ and $u_h^n$ be the fully-discretized variational form \eqref{varpara2disc}. Let $\Psi$ and $\Psi_h^n$ be the corresponding \sensfct, respectively given by \eqref{eq:2chainrule} and \eqref{eq:directheatsensitivity}. Then
  \small{
\begin{align*}
  & \forall n = 0,\dots,N_T,\  \norm{\Psi_h^n- \Psi(t)}_{L^2(\Omega)} \leq Ch^2\norm{\Psi^0}_{H^2(\Omega)} + h^2\big(C \int_0^{t^n} \norm{\Psi_t}_{H^2(\Omega)} \textrm{ ds } + C(\mu) \big[\int_0^{t^n} \norm{ u_t}_{H^2(\Omega)}^2 \textrm{ ds} \big]^{1/2}\big) \\
    & \quad \quad \quad  \quad \quad \quad  \quad \quad \quad \quad \quad \quad \quad \quad \quad \quad \quad + \Delta t_F \big( C \int_0^{t^n} \norm{ \Psi_{tt}}_{L^2(\Omega)} \textrm{ ds}  + C(\mu) \big[\int_0^{t^n} \norm{u_{tt}}_{L^2(\Omega)}^2 \textrm{ ds} \big]^{1/2} \big).
\end{align*}
}
\end{theorem}
{}{The proof of Theorem \ref{corollaryPsii} is detailed in the appendix \ref{proofEuler}.}\\
With the fully-discretized version \eqref{eq:directheatsensitivity}, the following estimate holds with the $H^1_{{}{0}}$ norm.
\begin{theorem}
  \label{corollaryPsiiH1disc}
  Let $\Omega$ be a convex polyhedron. Let $A(\mu)=\mu \ I_d$, with $\mu \in \mathbb{R}_*^+$.\\ Consider $u \in H^{1}(0,T;\ H^2(\Omega)) \cap H^{2}(0,T;\ L^2(\Omega))$ {}{to} be the solution of \eqref{heatEq2} with $u^0 \in H^2(\Omega)$ and $u_h^n$ be the fully-discretized variational form \eqref{varpara2disc}. Let $\Psi$ and $\Psi_h^n$ be the corresponding \sensfct, respectively given by \eqref{eq:2chainrule} and \eqref{eq:directheatsensitivity}. Then
\begin{align*}
 & \forall n=0,\dots, N_T,\ \norm{\nabla \Psi_h^n- \nabla \Psi(t)}_{L^2(\Omega)} \leq   h\big[C\norm{\Psi^0}_{H^2(\Omega)} + C(\mu) \int_0^{t^n} \norm{ \Psi_t}_{H^2(\Omega)} \textrm{ ds } + C(\mu) \big[\int_0^{t^n} \norm{u_t}_{H^2(\Omega)}^2 \textrm{ ds}\big]^{1/2}\big] \\
    &\quad \quad \quad  \quad \quad \quad  \quad \quad \quad \quad \quad \quad \quad \quad \quad \quad \quad \quad + C(\mu) \Delta t_F \big( \int_0^{t^n} \norm{ \Psi_{tt}}_{L^2(\Omega)} \textrm{ ds}  +  \big[\int_0^{t^n} \norm{ u_{tt}}_{L^2(\Omega)}^2 \textrm{ ds}\big]^{1/2} \big)\big].
\end{align*}
\end{theorem}
\begin{proof}
{}{In the proofs, we denote $A \lesssim B$ for $A \leq CB$ with $C$ not depending on $h$ nor $H$ and independent of the time steps. }
 As in \cite{thomee2}, we first decompose the error with two components $\theta^n$ and $\rho^n$  such that, on the discretized time grid $(t^n)_{n=0,\dots,N_T}$,
\begin{align}
\forall n=0,\dots,N_T, \ e^n:= \Psi_h^n- \Psi(t^n)&=(\Psi_h^n- P^1_h \Psi(t^n))+(P^1_h \Psi(t^n)- \Psi(t^n)), \nonumber \\
&:=\theta^n + \rho^n.
\label{rhonthetan}
\end{align}
\begin{itemize}
\item The estimate on $\rho^n$ is classical (see Lemma 1.1 \cite{thomee2}) and leads to
\begin{equation}
  \norm{\nabla \rho^n}_{L^2(\Omega)} \lesssim h \norm{\Psi(t^n)}_{H^2(\Omega)} \leq h \big[\norm{\Psi^0}_{H^2(\Omega)} + \int_0^{t^n} \norm{ \Psi_t}_{H^2(\Omega)} \ \textrm{ds} \big], \  \forall n=0,\dots,N_T.
\label{rhoclassiqueH1bis}
\end{equation}
\item For the estimate of $\theta^n$, {}{instead of choosing $v=\theta^n$ as in the proof of Theorem \ref{corollaryPsii} in Equation \eqref{secondexpression}, we take $v=\overline{\pt} \theta^n$ and} we have
  \begin{equation}
  \label{secondexpressionH1}
  \norm{\overline{\pt} \theta^n}^2_{L^2(\Omega)}+\mu (\nabla \theta^n,\nabla \overline{\pt} \theta^n)=\frac{1}{\mu}(\overline{\pt} u_h^n- u_t(t^n),\overline{\pt} \theta^n) - (w^n, \overline{\pt}\theta^n),
  \end{equation}
{}{with }
\begin{equation*}
{}{w^n:=w_1^n+w_2^n, \textrm{with }w_1^n:=(P_h^1-I)\overline{\pt} \Psi(t^n), \textrm{\quad   } w_2^n:=\overline{\pt} \Psi(t^n) - \Psi_t(t^n),}
  \end{equation*}

  which entails
    \begin{equation*}
 \mu (\nabla \theta^n,\nabla \overline{\pt} \theta^n) \leq \frac{1}{2\mu^2}\norm{\overline{\pt} u_h^n-  u_t(t^n)}_{L^2(\Omega)}^2  + \frac{1}{2}\norm{w^n}_{L^2(\Omega)}^2.
    \end{equation*}
By definition of $\overline{\pt}$
    \begin{equation*}
         \mu \norm{\nabla \theta^n}_{L^2(\Omega)}^2 \leq (\sqrt{\mu}\nabla \theta^n,\sqrt{\mu}\nabla \theta^{n-1}) + \frac{\Delta t_F}{2\mu^2}\norm{\overline{\pt} u_h^n- u_t(t^n)}_{L^2(\Omega)}^2  + \frac{\Delta t_F}{2}\norm{w^n}_{L^2(\Omega)}^2,
    \end{equation*}
       and by Young's inequality
            \begin{equation*}
         \mu \norm{\nabla \theta^n}_{L^2(\Omega)}^2 \leq  \mu\norm{\nabla \theta^{n-1}}_{L^2(\Omega)}^2 + \frac{\Delta t_F}{\mu^2}\norm{\overline{\pt} u_h^n-  u_t(t^n)}_{L^2(\Omega)}^2  + \Delta t_F \norm{w^n}_{L^2(\Omega)}^2,
    \end{equation*}
   which, by recursion, gives
        \begin{equation}
        \label{gradtheta}
 \mu \norm{\nabla \theta^n}_{L^2(\Omega)}^2 \leq \underbrace{ \frac{\Delta t_F}{\mu^2} \underset{j=1}{\overset{n}{\sum}} \norm{\overline{\pt} u_h^j- u_t(t^j)}_{L^2(\Omega)}^2 }_{{}{T_{1,n}'}}  + \underbrace{\Delta t_F  \underset{j=1}{\overset{n}{\sum}} \norm{w^j}_{L^2(\Omega)}^2}_{{}{T_{2,n}'}}.
    \end{equation}
    \begin{itemize} 
        \item {}{As for the previous estimate (of Theorem \ref{corollaryPsii}), detailed in Appendix \ref{proofEuler}, a bound of $T_{1,n}'$ is derived from \eqref{newdecompositiondisc}, \eqref{w1L2} and \eqref{w2L2}  (here, we consider an additional factor $\frac{1}{\mu}$), which yields }
\begin{equation*}
 T_{1,n}' \lesssim \frac{1}{\mu^2} \big[ h^4 \int_0^{t^n} \norm{u_t}_{H^2(\Omega)}^2 \textrm{ ds } + \Delta t_F^2 \int_0^{t^n} \norm{u_{tt}}_{L^2(\Omega)}^2 \textrm{ ds } \big].
\end{equation*}
\item To find a bound for $T_{2,n}'$, we simply use \eqref{w1L2} and \eqref{w2L2} again but with the sensitivity function $\Psi$ instead of the state function $u$, {}{and thus, we obtain
\begin{equation*}
 T_{2,n}' \lesssim  h^4 \int_0^{t^n} \norm{u_t}_{H^2(\Omega)}^2 \textrm{ ds } + \Delta t_F^2 \int_0^{t^n} \norm{u_{tt}}_{L^2(\Omega)}^2 \textrm{ ds } .
\end{equation*}
}
  \end{itemize}
  Combining the estimates on \eqref{gradtheta} (with the estimates on $T_{1,n}'$ and $T_{2,n}'$), and \eqref{rhoclassiqueH1bis} concludes the proof.
    \end{itemize}
\end{proof}
{}{\begin{remark}
  With the assumptions of Theorem \ref{corollaryPsiiH1disc}, we remark that the estimate, that we have derived, is optimal. Indeed, we adapted an optimal result from \cite{thomee2} where the rates of convergence are in $\mathcal{O}(h+\Delta t_F)$. Compared to the estimate of \cite{thomee2}, we now have an additional term in \eqref{gradtheta}, that we have named $T_{1,n}'$ and which takes into account the additional dependence of the source terms in \eqref{eq:directheatsensitivity} with the state solution. We showed in Theorem \ref{corollaryPsiiH1disc} that it does not degrade these rates, up to some constants that are independent of spatial and temporal grid steps. 
\end{remark}}
With $\overline{\pt} \Psi^m_H=\frac{\Psi^m_H-\Psi^{m-1}_H}{\Delta t_G}$, on the coarse mesh $\mathcal{T}_H$ with the Crank-Nicolson scheme, the fully-discretized system \eqref{varpara2discCN} yields
\begin{numcases}{}
  \label{eq:directheatsensitivityCN}
  &  Find $\Psi_{p,H}^m \in V_H$ for $m \in \{0,\dots, M_T\}$ such that \nonumber \\
&$(\overline{\pt} \Psi_{p,H}^m,v_H)+a\big(\frac{\Psi_{p,H}^m+\Psi_{p,H}^{m-1}}{2},v_H;\mu\big)= - \big(\frac{\pt A}{\pt \mu_p}(\mu) \frac{\nabla u_{H}^m(\mu) + \nabla u_{H}^{m-1}(\mu)}{2} , \nabla v_H\big),  \textrm{ for } v_H \in V_H, \textrm{ for } m=\{1,\dots,M_T\}$,  \\
& $\Psi_{p,H}^0(\cdot;\mu) = P^1_H \Psi^0_p (\cdot;\mu). \notag $
\end{numcases}
We have the following result in the $L^2$ norm with the Crank-Nicolson scheme on the coarse mesh $\mathcal{T}_H$.
\begin{theorem}
  \label{corollaryPsiiCrankNicolson}
  Let $\Omega$ be a convex polyhedron. Let $A(\mu)=\mu \ I_d$, with $\mu \in \mathbb{R}_*^+$. \\
  Consider $u \in H^{2}(0,T;\ H^2(\Omega)) \cap  H^{3}(0,T;\ L^2(\Omega)) $ {}{to} be the solution of \eqref{heatEq2} with $u^0 \in H^2(\Omega)$ and $u_H^m$ be the fully-discretized variational form \eqref{varpara2discCN} (on the coarse mesh $\mathcal{T}_H$). Let $\Psi$ and $\Psi_H^m$ be the corresponding \sensfct, respectively given by \eqref{eq:2chainrule} and \eqref{eq:directheatsensitivity}. Then
  \small{
\begin{align*}
  & \forall m=0,\dots,M_T, \norm{\Psi_{{}{H}}^m- \Psi(\widetilde{t}^m)}_{L^2(\Omega)} \leq CH^2 \big[\norm{\Psi^0}_{H^2(\Omega)} + \int_0^{\widetilde{t}^m} \norm{\Psi_t}_{H^2(\Omega)} \textrm{ ds } + C(\mu) \big[\int_0^{\widetilde{t}^m} \norm{u_t}_{H^2(\Omega)}^2 \textrm{ ds } \big]^{1/2}\big] \\
    &\quad \quad \quad \quad \quad \quad \quad \quad \quad + C \Delta t_G^2 \big( \int_0^{\widetilde{t}^m}\norm{ \Psi_{ttt}}_{L^2(\Omega)} \textrm{ ds}+ \big[ \int_0^{\widetilde{t}^m} \norm{\Delta u_{tt}}_{L^2(\Omega)}^2 \textrm{ ds}\big]^{1/2} + C(\mu) \big[ \big[\int_0^{\widetilde{t}^m} \norm{u_{ttt}}_{L^2(\Omega)}^2 \textrm{ ds}]^{1/2} + \int_0^{\widetilde{t}^m} \norm{ \Delta \Psi_{tt}}_{L^2(\Omega)} \textrm{ ds}\big] \big).
\end{align*}
}
\end{theorem}
\begin{proof}
In analogy with \eqref{rhonthetan}, the same definitions of $\rho^m$ and $\theta^m$ is used but with the coarse grids. 
  \begin{itemize}
    \item For $\rho^m$ we get 
    \begin{equation}
  \norm{\rho^m}_{L^2(\Omega)}\lesssim H^2 \big[\norm{\Psi^0}_{H^2(\Omega)} + \int_0^{\widetilde{t}^m} \norm{ \Psi_t}_{H^2(\Omega)} \ \textrm{ds} \big], \  \forall m \in \{0,\dots,M_T\}.
\label{rhoclassiquediscCN}
\end{equation}

  \item Then, we introduce the following notation
  \begin{equation}
    \label{newnotation}
    \widehat{u_H^m}:=\frac{1}{2}(u_H^m+u_H^{m-1}).
  \end{equation}
  Thanks to the Crank-Nicolson formulation on $\Psi_H^m$ \eqref{eq:directheatsensitivityCN} and $u_H^m$ \eqref{varpara2discCN} on the coarse mesh $\mathcal{T}_H$, 
   \begin{align*}
  (\overline{\pt} \theta^m,v)+\mu (\nabla \wHat{\theta^m},\nabla v) &=(\overline{\pt} \Psi_H^m,v)-(\overline{\pt} P_H^1(\Psi(\widetilde{t}^m)),v) + \mu (\nabla \wHat{\Psi_H^m},\nabla v) - \frac{\mu}{2}\big((\nabla P_H^1 \Psi(\widetilde{t}^m),\nabla v) + (\nabla P_H^1 \Psi(\widetilde{t}^{m-1}),\nabla v) \big),\\
  & = -(\nabla \wHat{u_H^m} ,\nabla v)- (\overline{\pt} P_H^1(\Psi(\widetilde{t}^m)),v) - \frac{\mu}{2}\big((\nabla \Psi(\widetilde{t}^m),\nabla v) + (\nabla  \Psi(\widetilde{t}^{m-1}),\nabla v) \big), \\
  &=-(\nabla \wHat{u_H^m} ,\nabla v) - w_I^m[\Psi] -w_{II}^m[\Psi]  - ( \Psi_t(\widetilde{t}^{m-\frac{1}{2}}),v) - \frac{\mu}{2}\big((\nabla  \Psi(\widetilde{t}^m),\nabla v) + (\nabla \Psi(\widetilde{t}^{m-1}),\nabla v) \big), \\
  &= (-\nabla \wHat{u_H^m} + \nabla u(\widetilde{t}^{m-\frac{1}{2}}),\nabla v) -w_I^m -w_{II}^m +  \mu(\nabla \Psi(\widetilde{t}^{m-\frac{1}{2}}),\nabla v)  -  \frac{\mu}{2}\big(\nabla  \Psi(\widetilde{t}^m)+\nabla \Psi(\widetilde{t}^{m-1}),\nabla v \big),\\
  &=(\nabla u(\widetilde{t}^{m-\frac{1}{2}}) - \wHat{\nabla u_H^m},\nabla v) -(\underbrace{w_I^m[\Psi] +w_{II}^m[\Psi] + \mu w_{III}^m[\Psi]}_{w_T^m[\Psi]},v),
\end{align*}
where $w_I^m$, $w_{II}^m$ and $w_{III}^m$ are defined by
\begin{equation}
  \label{w1w2w3}
  w_I^m[\Psi]:=(P_H^1-I)\overline{\pt} \Psi(\widetilde{t}^m), \textrm{ } w_{II}^m[\Psi]:=\overline{\pt} \Psi(\widetilde{t}^m) -  \Psi_t(\widetilde{t}^{m-\frac{1}{2}}) ,  \textrm{ and } w_{III}^m[\Psi]:=\Delta \Psi(\widetilde{t}^{m-\frac{1}{2}}) - \frac{1}{2}(\Delta \Psi (\widetilde{t}^m)+\Delta \Psi(\widetilde{t}^{m-1})).
\end{equation}
 Thus, Equation \eqref{secondexpression}, with a Crank-Nicolson scheme and with $v=\wHat{\theta^m}$, becomes
  \begin{equation}
  \label{secondexpressionbis}
  (\overline{\pt} \theta^m,\wHat{\theta^m})+\mu (\nabla \wHat{\theta^m},\nabla \wHat{\theta^m}) 
  =\frac{1}{\mu}(\overline{\pt} u_H^m-  u_t(\widetilde{t}^{m-\frac{1}{2}}), \wHat{\theta^m}) - (w_T^m[\Psi], \wHat{\theta^m}),
\end{equation}
where $w_T^m[\Psi]=w_I^m[\Psi]+w_{II}^m[\Psi]+\mu w_{III}^m[\Psi]$.
By definition of $\overline{\pt}$ (with the coarse time grid), and since the second term in \eqref{secondexpressionbis} is always non-negative, we get
 \begin{equation*}
  ( \theta^m,\wHat{\theta^m}) - ( \theta^{m-1},\wHat{\theta^m}) \leq \Delta t_G \big[\frac{1}{\mu}\norm{\overline{\pt} u_H^m-  u_t(\widetilde{t}^{m-\frac{1}{2}})}_{L^2(\Omega)} + \norm{w_T^m[\Psi]}_{L^2(\Omega)}\big] \norm{\wHat{\theta^m}}_{L^2(\Omega)},
 \end{equation*}
 and by definition of $\wHat{\theta^m}$ \eqref{newnotation},
  \begin{equation*}
  \norm{ \theta^m }_{L^2(\Omega)}^2 -  \norm{\theta^{m-1}}_{L^2(\Omega)}^2 \leq \Delta t_G \big[\frac{1}{\mu}\norm{\overline{\pt} u_h^m-  u_t(\widetilde{t}^{m-\frac{1}{2}})}_{L^2(\Omega)} + \norm{w_T^m[\Psi]}_{L^2(\Omega)}\big] \norm{\theta^m + \theta^{m-1}}_{L^2(\Omega)},
  \end{equation*}
  so that, after cancellation of a common factor,
    \begin{equation*}
  \norm{ \theta^m }_{L^2(\Omega)} -  \norm{\theta^{m-1}}_{L^2(\Omega)} \leq \Delta t_G \big[\frac{1}{\mu}\norm{\overline{\pt} u_H^m-  u_t(\widetilde{t}^{m-\frac{1}{2}})}_{L^2(\Omega)} + \norm{w_T^m[\Psi]}_{L^2(\Omega)}\big],
    \end{equation*}
    and by recursive application, it entails
       \begin{equation}
       \label{T2T3''}
  \norm{ \theta^m }_{L^2(\Omega)} \leq \underbrace{ \frac{\Delta t_G}{\mu} \overset{m}{\underset{j=1}{\sum}}\norm{\overline{\pt} u_H^j-  u_t(\widetilde{t}^{j-\frac{1}{2}})}_{L^2(\Omega)}}_{{}{T_{1,n}''}} +\underbrace{\Delta t_G \overset{m}{\underset{j=1}{\sum}} \norm{w_T^j[\Psi]}_{L^2(\Omega)}}_{{}{T_{2,n}''}}.
    \end{equation}
\begin{itemize}     
\item To estimate $T_{1,n}''$, we use the same tricks as {}{in the appendix, Equation \eqref{decompositiondisc}}. 
       First, we decompose $T_{1,n}''$ in two contributions, such that 
       \begin{equation*}
            \frac{\Delta t_G}{\mu} \underset{j=1}{\overset{m}{\sum}} \norm{\overline{\pt} u_H^j -  u_t(\widetilde{t}^{j-\frac{1}{2}})}_{L^2(\Omega)} \leq   \frac{\Delta t_G}{\mu} \underset{j=1}{\overset{m}{\sum}} \underbrace{\norm{\overline{\pt} u_H^j - \overline{\pt} P_1^h u(\widetilde{t}^j)}_{L^2(\Omega)}}_{\norm{\overline{\pt} \theta_u^j}_{L^2(\Omega)}} +\underbrace{ \norm{\overline{\pt} P_1^h u(\widetilde{t}^j) -  u_t(\widetilde{t}^{j-\frac{1}{2}})}_{L^2(\Omega)}}_{\norm{w_{I}^j[u] + w_{II}^j[u]}_{L^2(\Omega)}},
       \end{equation*}
     {}{  where $\theta_u^m$ is the coarse version of \eqref{errorParabolU}.}
    From the Cauchy-Schwarz inequality
       \begin{align} 
  \label{decompositiondiscCK}
  \frac{\Delta t_G}{\mu} \underset{j=1}{\overset{m}{\sum}} \norm{\overline{\pt} u_H^j -  u_t(\widetilde{t}^{j-\frac{1}{2}})}_{L^2(\Omega)}  &\leq \frac{(\widetilde{t}^m-\widetilde{t}^0)^{1/2}}{\mu} \big[\big( \underset{j=1}{\overset{m}{\sum}} \Delta t_G \norm{\overline{\pt} \theta^j_u}_{L^2(\Omega)}^2 \big)^{1/2}+ \big( \underset{j=1}{\overset{m}{\sum}} \Delta t_G \norm{w^j_{I}[u]+w^j_{II}[u]}_{L^2(\Omega)}^2  \big)^{1/2}\big], \nonumber \\
  &\lesssim \frac{\Delta t_G^{1/2}}{\mu} \big[\big( \underset{j=1}{\overset{m}{\sum}} \norm{\overline{\pt} \theta^j_u}_{L^2(\Omega)}^2 \big)^{1/2} \nonumber \\
  & \quad \quad \quad \quad \quad   + \big( \underset{j=1}{\overset{m}{\sum}} \norm{w^j_{I}[u]+w^j_{II}[u]}_{L^2(\Omega)}^2 + \norm{w^j_{III}[u]}_{L^2(\Omega)}^2 \big)^{1/2}\big]. 
       \end{align}
  \begin{itemize}
      \item To estimate the first term of \eqref{decompositiondiscCK}, we choose $v= \overline{\pt}\theta^m_u,$ and from the Crank-Nicolson scheme \eqref{varpara2discCN},
  \begin{equation*}
    \norm{\overline{\pt} \theta_u^m}_{L^2(\Omega)}^2 + \mu (\nabla \wHat{\theta_u^m},\nabla \overline{\pt} \theta_u^m) = -(w_{T}^m[u], \overline{\pt} \theta^m_u).
  \end{equation*}
 By definitions of $\overline{\pt}$ and $\wHat{\theta_u^m}$, it can be rewritten
    \begin{equation*}
      \norm{\overline{\pt} \theta_u^m}_{L^2(\Omega)}^2 + \frac{\mu}{2\Delta t_G} \norm{\nabla \theta_u^m}_{L^2(\Omega)}^2 -  \frac{\mu}{2 \Delta t_F } \norm{\nabla \theta_u^{m-1}}_{L^2(\Omega)}^2= -(w_{T}^m[u], \overline{\pt} \theta^m_u),
    \end{equation*}
  and with Young's inequality, we get
    \begin{equation}
      \label{pourbound}
      \norm{\overline{\pt} \theta_u^m}_{L^2(\Omega)}^2 + \frac{\mu}{\Delta t_G} \norm{\nabla \theta_u^{m}}_{L^2(\Omega)}^2 - \frac{\mu}{\Delta t_G} \norm{\nabla \theta_u^{m-1}}_{L^2(\Omega)}^2 \leq \norm{ w_{T}^m[u]}_{L^2(\Omega)}^2.
        \end{equation}
        Now by summing over all time steps to obtain a telescoping sum, we find that
        \begin{equation*}
  \underset{j=1}{\overset{m}{\sum}}  \norm{\overline{\pt} \theta^j_u}_{L^2(\Omega)}^2 \leq \frac{\mu}{\Delta t_G} \norm{\nabla \theta_u^m}_{L^2(\Omega)}^2 +  \underset{j=1}{\overset{m}{\sum}}\norm{ w_{T}^j[u]}_{L^2(\Omega)}^2,
        \end{equation*}
     and by repeated application using \eqref{pourbound} to bound the first right-hand side term,
      we find that
                \begin{equation*}
  \underset{j=1}{\overset{m}{\sum}}  \norm{\overline{\pt} \theta^j_u}_{L^2(\Omega)}^2 \leq 2 \underset{j=1}{\overset{m}{\sum}} \norm{w_{I}^j[u] + w_{II}^j[u] + \mu w_{III}^j[u]}_{L^2(\Omega)}^2 \leq 4 \underset{j=1}{\overset{m}{\sum}} \norm{w_{I}^j[u] + w_{II}^j[u] }_{L^2(\Omega)}^2 +\norm{ \mu w_{III}^j[u]}_{L^2(\Omega)}^2  ,
        \end{equation*}
        therefore we obtain for \eqref{decompositiondiscCK}
        \begin{equation} 
          \label{decompositiondiscCK2}
  \frac{\Delta t_G}{\mu} \underset{j=1}{\overset{m}{\sum}} \norm{\overline{\pt} u_H^j -  u_t(\widetilde{t}^{j-\frac{1}{2}})}_{L^2(\Omega)}  \lesssim \Delta t_G^{1/2}\big[\big(  \underset{j=1}{\overset{m}{\sum}} \frac{1}{\mu^2} \norm{w^j_{I}[u]+w^j_{II}[u]}_{L^2(\Omega)}^2+ \norm{w^j_{III}[u]}_{L^2(\Omega)}^2\big)^{1/2}\big].
       \end{equation}
       \item       Now, we can estimate the right-hand side terms of \eqref{decompositiondiscCK2} as done in \cite{thomee2}.
       We first remark that $w_{I}^j[u]=w_{1,u}^j$ (and the estimate is given by \eqref{w1L2} but with the coarse spatial and time grids),
       so it remains to seek bounds for $w_{II}^j[u]$ and $w_{III}^j[u]$.
       \begin{itemize}
       \item For $w_{II}^j[u]$,
       \begin{align}
         \label{wIIu}
         \Delta t_G \underset{j=1}{\overset{m}{\sum}}  \norm{w_{II}^j[u]}_{L^2(\Omega)}^2 &= \frac{1}{\Delta t_G}\underset{j=1}{\overset{m}{\sum}}  \norm{u(\widetilde{t}^j)-u(\widetilde{t}^{j-1})-\Delta t_G \  u_t(\widetilde{t}^{j-\frac{1}{2}})}_{L^2(\Omega)}^2 \notag, \\
         &=  \frac{1}{4 \Delta t_G } \underset{j=1}{\overset{m}{\sum}}  \norm{\int_{\widetilde{t}^{j-1}}^{\widetilde{t}^{j-\frac{1}{2}}} (s-\widetilde{t}^{j-1})^2  u_{ttt}(s) + \int_{\widetilde{t}^{j-\frac{1}{2}}}^{\widetilde{t}^{j}} (s-\widetilde{t}^{j})^2 u_{ttt}(s) \textrm{ ds}}_{L^2(\Omega)}^2\notag \\
         &\lesssim  \Delta t_G^3 \underset{j=1}{\overset{m}{\sum}}  \norm{  \int_{\widetilde{t}^{j-1}}^{\widetilde{t}^j} u_{ttt}(s) \textrm{ ds} }_{L^2(\Omega)}^2,\nonumber \\
         & \leq  \Delta t_G^4 \underset{j=1}{\overset{m}{\sum}}  \int_{\widetilde{t}^{j-1}}^{\widetilde{t}^j} \norm{u_{ttt} }^2_{ L^2(\Omega)} \textrm{ ds} \leq  \Delta t_G^4 \int_{\widetilde{t}^0}^{\widetilde{t}^m} \norm{u_{ttt}}^2_{ L^2(\Omega)} \textrm{ ds}.
       \end{align}
       
       \item For $w_{III}^j[u]$,
       \begin{align}
         \label{wIIIu}
         \Delta t_G \underset{j=1}{\overset{m}{\sum}}  \norm{w_{III}^j[u]}_{L^2(\Omega)}^2 &=\Delta t_G \underset{j=1}{\overset{m}{\sum}}  \norm{\Delta u(\widetilde{t}^{j-\frac{1}{2}})-\frac{1}{2}(\Delta u(\widetilde{t}^{j}) +\Delta u(\widetilde{t}^{j-1}))}_{L^2(\Omega)}^2, \nonumber \\
         &=\frac{\Delta t_G }{4}\underset{j=1}{\overset{m}{\sum}}  \norm{\int_{\widetilde{t}^{j-1}}^{\widetilde{t}^{j-\frac{1}{2}}} (t^{j-1}-s)\Delta u_{tt}(s) \textrm{ ds } + \int_{\widetilde{t}^{j-\frac{1}{2}}}^{\widetilde{t}^{j}} (s-\widetilde{t}^j)\Delta u_{tt}(s) \textrm{ ds } }_{L^2(\Omega)}^2, \nonumber \\
         &\lesssim \Delta t_G ^3\underset{j=1}{\overset{m}{\sum}} \norm{ \int_{\widetilde{t}^{j-1}}^{\widetilde{t}^j} \Delta u_{tt}  \textrm{ ds} }_{ L^2(\Omega)}^2 \leq \Delta t_G^4 \int_{\widetilde{t}^0}^{\widetilde{t}^{m}} \norm{\Delta u_{tt}  }_{ L^2(\Omega)}^2 \textrm{ ds}.
       \end{align}
       \end{itemize}
       Altogether,
       \begin{equation}
       \label{T2n''}
        T_{1,n}'' \lesssim \big[\frac{H^2}{\mu} \big(\int_0^{\widetilde{t}^m}  \norm{u_t}_{H^2(\Omega)}^2 \textrm{ ds}\big)^{1/2}  + \Delta t_G^2 \big(\big( \frac{1}{\mu}\int_0^{\widetilde{t}^m} \norm{u_{ttt}}_{ L^2(\Omega)}^2 \big)^{1/2}   +  \big( \int_0^{\widetilde{t}^m} \norm{\Delta u_{tt} }_{ L^2(\Omega)}^2 \big)^{1/2} \textrm{ ds} \big)\big].
       \end{equation}
         \end{itemize}
         \item   To estimate $T_{2,n}''$ defined in \eqref{T2T3''}, we remark that $w_I^j=w_{1}^j$ (but with the coarse spatial and time grids),
       and 
       \begin{itemize}
       \item for $w_{II}^j$,
       \begin{align}
         \label{wII}
         \Delta t_G \norm{w_{II}^j}_{L^2(\Omega)} &\leq \norm{\Psi(\widetilde{t}^j)-\Psi(\widetilde{t}^{j-1})-\Delta t_G \Psi_t(\widetilde{t}^{j-\frac{1}{2}})}_{L^2(\Omega)} \notag, \\
         &=  \frac{1}{2} \norm{\int_{\widetilde{t}^{j-1}}^{\widetilde{t}^{j-\frac{1}{2}}} (s-\widetilde{t}^{j-1})^2 \Psi_{ttt}(s) + \int_{\widetilde{t}^{j-\frac{1}{2}}}^{\widetilde{t}^{j}} (s-\widetilde{t}^{j})^2 \Psi_{ttt} (s) \textrm{ ds}}_{L^2(\Omega)}\notag \\
         &\lesssim \Delta t_G^2 \int_{\widetilde{t}^{j-1}}^{\widetilde{t}^j} \norm{\Psi_{ttt}}_{ L^2(\Omega)} \textrm{ ds}.
       \end{align}
       \item Finally, for $w_{III}^j$,
       \begin{equation}
         \label{wIII}
         \Delta t_G \norm{w_{III}^j}_{L^2(\Omega)} =\Delta t_G \norm{\Psi(\widetilde{t}^{j-\frac{1}{2}})-\frac{1}{2}(\Psi(\widetilde{t}^{j}) + \Psi(\widetilde{t}^{j-1}))}_{L^2(\Omega)} \lesssim \Delta t_G^2 \int_{\widetilde{t}^{j-1}}^{\widetilde{t}^j} \norm{\Delta \Psi_{tt} }_{ L^2(\Omega)} \textrm{ ds}.
       \end{equation}
       \end{itemize}
       Altogether, 
       \begin{equation}
       \label{T3n''}
         T_{2,n}'' \lesssim H^2 \int_0^{\widetilde{t}^m} \norm{\Psi_t}_{H^2(\Omega)} \textrm{ ds}  + \Delta t_G^2 \int_0^{\widetilde{t}^m} \big(\norm{\Psi_{ttt} }_{ L^2(\Omega)} + \mu \norm{\Delta \Psi_{tt}}_{ L^2(\Omega)} \big) \textrm{ ds},
       \end{equation}
       \end{itemize}
        which concludes the proof (combining \eqref{T2T3''}, \eqref{T2n''} and \eqref{T3n''}).
\end{itemize}
\end{proof}

In analogy with the previous work on parabolic equations \cite{paraboliccontext}, we define \begin{equation}
\label{parabola2}
  \widetilde{\Psi_H}^n=I^2_n[\Psi_H^m](\mu), \quad \textrm{for }n=0,\dots,N_T, 
    \end{equation}
with $I^2_n$ defined by \eqref{parabola} as the {}{piecewise} quadratic interpolation in time of the coarse solution at the fine times.
\begin{corollary}[of Theorem \ref{corollaryPsiiCrankNicolson}]
  \label{corollaryCorollaryPsiiCrankNicolson}
  Under the assumptions of Theorem \ref{corollaryPsiiCrankNicolson}, let $u_H^m$ be the fully-discretized solution \eqref{varpara2discCN} on the coarse mesh $\mathcal{T}_H$. Let $\Psi$ and $\Psi_H^m$ be the corresponding sensitivities, \ respectively given by \eqref{eq:2chainrule} and by \eqref{eq:directheatsensitivityCN}. Let $\widetilde{\Psi_H}^n$ be the {}{piecewise} quadratic interpolation of the coarse solution $\Psi_H^m$ given by \eqref{parabola2}. Then,
    \small{
  \begin{align*}
  & \forall n=0,\dots,N_T,\ \norm{\widetilde{\Psi_{{}{H}}}^n- \Psi(t^n)}_{L^2(\Omega)} \leq CH^2\big[\norm{\Psi^0}_{H^2(\Omega)} +  \int_0^{t^n} \norm{\Psi_t}_{H^2(\Omega)} \textrm{ ds } + C(\mu) \big[\int_0^{t^n} \norm{u_t}_{H^2(\Omega)}^2 \textrm{ ds }  \big]^{1/2}\big] \\
    &\quad \quad \quad \quad \quad \quad \quad \quad \quad + C \Delta t_G^2 \big(  \int_0^{t^n}\norm{ \Psi_{ttt}}_{L^2(\Omega)} \textrm{ ds }  + \big[  \int_0^{t^n}\norm{\Delta u_{tt}}_{L^2(\Omega)}^2 \textrm{ ds } \big]^{1/2} + C(\mu) \big[ \big[ \int_0^{t^n}\norm{u_{ttt}}_{L^2(\Omega)}^2 \textrm{ ds }  ]^{1/2} +  \int_0^{t^n}  \norm{ \Delta \Psi_{tt}}_{L^2(\Omega)}\textrm{ ds }  \big] \big).
\end{align*}
}
\end{corollary}
In the next section, we proceed with the adjoint state formulation.
\subsection{Sensitivity analysis: The adjoint problem.}
The adjoint {}{can result } from an inverse method, where we aim at retrieving the optimal parameter of an objective function $\mathcal{F}$. The latter will have a different meaning whether the goal is to retrieve the parameters from several {}{measurements} (for parameter identification) or if we want to optimize a function depending on the variables (PDE-constrained optimization).
{}{In this paper, we focus on an objective $\mathcal{F}$ with the following form} (in its fully-discretized version)
\begin{equation}
\label{Fobj}
    \mathcal{F}(\mu)= \frac{1}{2}\sum_{n=0}^{N_T} \underbrace{\norm{u_h^n(\mu)-\overline{u}^n}^2_{L^2(\Omega)}}_{{}{\norm{\mathrm{err}(t^n;\ \mu)}_{L^2(\Omega)}^2}},
\end{equation}
where the term $\overline{u}^n$ refers to the {}{measurements} which may be noisy (here for simplicity we consider the case of {}{measurements} on the variables although it may be given by other outputs). 
    Note that by differentiating $\mathcal{F}$ with respect to the parameters $\mu_p$, $p=1,\dots, P$,  we can observe the influence on the objective function of the input parameters through the normalized sensitivity coefficients (also called elasticity of $\mathcal{P}$) \cite{reviewSA}
\begin{equation*}
    S_k=\frac{\pt \mathcal{F}}{\pt \mu_k}(\mu)\times \frac{\mu_k}{\mathcal{F}(\mu)},\ k=1,\dots,P.
\end{equation*}
\subsubsection{The continuous setting.}
\begin{itemize}
  \item  Let us for instance consider the first case outlined above given in its continuous version by
\begin{equation}
\label{Fobjcontinuous}
    \mathcal{F}(\mu)=\frac{1}{2} \int_{0}^T \norm{\mathrm{err}(t;\ \mu)}^2_{L^2(\Omega)}.
\end{equation}
\item To minimize $\mathcal{F}$ under the constraint that "$u$ is the solution of our model problem \eqref{heatEq2}", we consider the following Lagrangian with $(\chi, \varphi)$ the {}{Lagrange multipliers}
  \begin{equation}
  \label{lagrangian}
    \mathcal{L}(u,\chi,\varphi,\zeta;\mu)=\mathcal{F}(\mu)+ \int_0^T \big[(\chi, (\nabla \cdot (A(\mu)\nabla u) + f - u_t))_{L^2(\Omega)} + (\varphi,u)_{L^2(\pt \Omega)} \big] \textrm{ ds}+ {}{ (\zeta,u_0(\mu)-u(0;\mu))_{L^2(\Omega)}},
  \end{equation}
  where \begin{itemize}
  \item $\chi \in H_0^1(\Omega)$ is the multiplier associated to the constraint ``$u$ is a solution of \eqref{heatEq2}'',
    \item $\varphi \in \mathbb{R}$ is the multiplier associated to the constraint of the Dirichlet boundary condition on $\pt \Omega$. Since we here consider {}{a} homogeneous condition, we just impose $\varphi=0$,
     \item {}{$\zeta \in H_0^1(\Omega)$ is the multiplier associated to the constraint ``$u(0;\mu)$ is the exact initial condition''. As for $\varphi$, we impose $\zeta=0.$}
  \end{itemize}
  \item  Differentiating $\mathcal{L}$ with respect to the parameter $\mu_p$, for $p=1,\dots,P$, we obtain the following adjoint system in its variational form
  (see {}{\cite{sykes1984adjoint,givoli2021tutorial}} for more details)
 \begin{numcases}{}
  \label{eq:chicontinuous}
  &  Find $\chi(t) \in V$ for $t \in [0,T]$ such that \nonumber \\
&$ ( \chi_t(t),v)= - (\frac{\pt err}{\pt u}(t;\mu),v) + (A(\mu) \nabla \chi(t),\nabla v), \ \forall v \in V, t < T$, \notag \\
      & $\chi(\cdot,T) = 0,\ \textrm{ in }\Omega$.
\end{numcases}
\item After solving \eqref{eq:chicontinuous} with the parameter $\mu$, one can compute the sensitivities of the objective function $\frac{d\mathcal{F}}{d\mu_p}$ by noticing that  {}{\cite{sykes1984adjoint,givoli2021tutorial}}
  \begin{equation}
  \label{Dmup1}
    \frac{d\mathcal{F}}{d\mu_p}= \frac{d\mathcal{L}}{d\mu_p}=\int_0^T \ \Big(\chi, \nabla \cdot (\frac{\pt A}{\pt \mu_p}(\mu) \nabla u )\big) \textrm{ ds}.
    \end{equation}
  \end{itemize}
  \subsubsection{Discretized setting.}
 In analogy with the direct method, we first discretize the system \eqref{eq:chicontinuous} in space, and then apply an Euler scheme backward in time with the fine grids and the Crank-Nicolson method with the coarse ones. The semi-discretized version on $\mathcal{T}_h$ {}{reads}
  \begin{numcases}{}
  \label{eq:chicontinuous2}
  &  Find $\chi_h(t) \in V_h$ for $t \in [0,T]$ such that \nonumber \\
&$ (\chi_{h,t}(t),v_h) - a(\chi_h,v_h;\mu)= - \big(\frac{\pt \textrm{err}_h}{\pt u_h}(t;\mu),v_h\big), \ \forall v_h \in V_h, t < T$, \notag \\
      & $\chi_h(\cdot,T) = 0,\ \textrm{ in }\Omega$.
\end{numcases}
With the fully-discretized version on the fine grids, the adjoint system \eqref{eq:chicontinuous2} becomes 
    \begin{numcases}{}
  &  Find $\chi_{h}^n \in V_h$ for $n \in \{0,\dots, N_T\}$ such that \nonumber \\
  & $ (\overline{\pt} \chi_{h}^n,v_h) - a(\chi_h^n,v_h;\mu) =-(u_h^n-\overline{u}^n,v_h), \ \forall n =0,\dots , N_T-1$, \nonumber \\
    & $ \chi_h^{N_T}(\cdot)\ =0 $.
    \label{adjointdisceulerweak}
    \end{numcases}
    Note that to compute $\frac{\pt \textrm{err}_h^n}{\pt u_h}$, we need the fine solutions $u_h^n$ and the associated {}{measurements}. 
   As for the state variable \eqref{varpara2discCN}, we also compute the adjoint on the coarse mesh with the Crank-Nicolson scheme,
       \begin{numcases}{}
  &  Find $\chi_{H}^m \in V_H$ for $m \in \{0,\dots, M_T\}$ such that \nonumber \\
  & $ (\overline{\pt} \chi_{H}^m,v_H) - a(\widehat{\chi^m_H},v_H;\mu) =-\frac{1}{2}\big( (u_H^m-\overline{u}^m,v_H)+(u_H^{m-1}-\overline{u}^{m-1},v_H) \big), \ \forall m =0,\dots , M_T-1$, \nonumber \\
    & $ \chi_H^{M_T}(\cdot)\ =0 $,
    \label{adjointdisceulerweakCN}
    \end{numcases}
    with  $\widehat{\chi_H^m}=\frac{1}{2}(\chi_H^m+\chi_H^{m-1}).$
   Finally, note that the problems \eqref{adjointdisceulerweak} and \eqref{adjointdisceulerweakCN} are well-posed, since they are solved backward in time (see \cite{givoli2021tutorial} for precisions in the general setting of time-dependent PDEs).\\
In the next section, we recall the NIRB algorithm in the context of parabolic equations. 
\subsection{ The Non-Intrusive Reduced Basis method (NIRB) in the context of parabolic equations.}
\label{NIRBParabolic}
Let $u(\mu)$ be the exact solution of problem \eqref{heatEq2} for a parameter  $\mu \in \mathcal{G}$. With the NIRB two-grid method, we aim at quickly approximating this solution by using a reduced space, denoted $X_h^N$, constructed from $N$ fully discretized solutions of \eqref{varpara2disc}, namely the so-called snapshots. We emphasize that the efficiency of this method relies on the assumption that its manifold solution $\mathcal{S} = \{u(\mu), \mu \in \mathcal{G} \}$ has a small Kolmogorov $N$-width \cite{kolmo}. Since each snapshot is a HF finite element approximation in space at a time $t^n$, $n = 0,..., N_T$ ($N_T$ potentially being very high), not all of the time steps may be required for the construction of the reduced space. Let $N$ refer to the number of snapshots employed for the RB construction and $N_{\mu}$ be the corresponding number of parameters used. For each parameter $\mu_i$, $i=1,\dots,N_{\mu},$ selected during the RB generation,  the associated number of time steps employed is denoted $N^i$.
Thus, the RB space is defined as 
\begin{equation}
    X_h^N:=\textrm{Span}\{u_h^{(n_j)_i}(\mu_i)| \ (n_j)_i \subset \{0,\cdots,N_T \}, \ i=1,\dots,N_{\mu}, \ j=1,\dots,N^i \}, \textrm{ with }N:=\overset{N_{\mu}}{\underset{i=1}{\sum}} N^i .
    \end{equation}
{}{Once the RB is created, the NIRB approximation, denoted $u_{Hh}^{N,n}$, takes a coarse FEM solution and projects it onto the RB space. As a result, the classical approximation has the following form: 
\begin{equation*}     
     \textrm{For }n=0,\dots, N_T, \quad   u_{Hh}^{N,n}(\mu):= \overset{N}{\underset{i=1}{\sum}}\ A_{i}^n(\mu)\Phi^{h}_{i},
\end{equation*}
where $A_{i}^n(\mu)$ are coefficients of the projection of the coarse FEM solution onto $X_h^N$. }

We now recall the offline/online decomposition of the NIRB two-grid procedure with parabolic equations and summarize it in Algorithms \ref{offline} and \ref{online}:
\begin{itemize}
\item \textbf{``Offline step''}\\
  The offline part of the algorithm allows us to construct the reduced space $X_h^N$.\\
  \begin{enumerate}
  \item  From training parameters $(\mu_i)_{i \in \{1,\dots , N_{train}\}}$, we define $\mathcal{G}_{train}= \underset{i \in \{1,\dots, N_{train}\}}{\cup} \mu_i$. Then, in order to construct the RB, we employ a greedy procedure which adequately chooses the parameters $(\mu_i)_{i=1,\dots,N_{\mu}}$ within $\mathcal{G}_{train}$. This procedure is described in Algorithm \ref{algogreedy} (with the setting $N_{\mu}=N$ to simplify notations). Note that a POD-greedy algorithm may also be employed \cite{paraboliccontext,haasdonk2008reduced,haasdonk2013convergence,knezevic2010certified}.
  
  \begin{algorithm}
     \caption{Greedy algorithm }  \label{algogreedy}
       \textbf{Input:} $tol$, \ $\{\uu_h^n(\mu_1),\cdots,\uu_h^n(\mu_{N_{train}}) \textrm{ with } \mu_i  \in \mathcal{G}_{train}, \ n=0,\dots, N_T \}.$ \\
    \textbf{Output:} Reduced basis $\{\Phi_1^h,\cdots,\Phi_{N}^h\}$. \\
    \begin{algorithmic}
      \State $\triangleright$ Choose $\mu_1,n_1=\underset{ \mu \in \mathcal{G}_{train},\  n \in \{ 0,\dots,N_T \}}{\textrm{arg} \max} \norm{\uu_h^n(\mu)}_{L^2(\Omega)}$ ,
      \State $\triangleright$  Set $\Phi_1^h=\frac{\uu_h^{n_1}(\mu_1)}{\norm{\uu_h^{n_1}(\mu_1)}_{L^2}}$ \; 
\State $\triangleright$ Set $ \mathcal{G}_1=\{\mu_1, n_1 \}$ and $X_h^1=\spann\{ \Phi_1^h \}$.
  \For{$k=2$ to $N$}:
  \State $\triangleright$ $\mu_k,n_k=$ arg $ \underset{(\mu,\ n) \in (\mathcal{G}_{train}\times \{0,\dots,N_T\}) \backslash \mathcal{G}_{k-1}}{ \textrm{max}} \norm{ \uu_h^n ( \mu ) - P^{k-1} ( \uu_h^n(\mu) ) }_{L^2}$,  with $P^{k-1}(u_h^n(\mu)):=\overset{k-1}{\underset{i=1}{\sum}} (u_h^n(\mu),\Phi_i^h)\ \Phi_i^k.$
  \State $\triangleright$ Compute $\widetilde{\Phi_k^h}=\uu_h^{n_k}(\mu_k)-P^{k-1}(u_h^{n_k}(\mu_k))$ and set $\Phi_k^h=\frac{\widetilde{\Phi^h_k}}{\norm{\widetilde{\Phi^h_k}}_{L^2(\Omega)}}$ 
    \State $\triangleright$ Set $\mathcal{G}_k=\mathcal{G}_{k-1}\cup \{ \mu_k\}$ and $X_h^k=X_h^{k-1} \oplus \spann\{\Phi^h_k \} $ 
   \State   $\triangleright$ Stop when $\norm{\uu_h^n(\mu)-P^{k-1}(\uu_h^n(\mu))}_{L^2}\leq tol, \ \forall \mu \in \mathcal{G}_{train},\ \forall n= 0,\dots, N_T. $ 
  \EndFor
  \end{algorithmic}
  \end{algorithm}
 
   The RB time-independent functions, denoted $(\Phi^{h}_{i})_{i=1,\dots,N}$, are generated at the end of this step from fine fully-discretized solutions $\{\uu_{h}^n(\mu_i)\}_{i \in \{1,\dots, N_{\mu}\},\ n=\{0,\dots,N_T\}}$  by solving the problem with the first order scheme (e.g. \eqref{varpara2disc} in case of the heat equation) with the HF solver. Note that even if all the time steps are computed, only $N^i$ are used for each $i \in \{1,\dots,N_{\mu} \}$ in the RB construction. Since at each step $k$ of the procedure, all sets added in the RB are in the orthogonal complement of $X_h^{k-1}$, it yields an $L^2$ orthogonal basis without further processing. Hence, $X_h^{N}$ can be defined as $X_h^{N}=\textrm{Span} \{ \Phi^h_1, \dots ,  \Phi^h_N\}$, with $(\Phi_h^i)_{i=1,\dots,N}$ $L^2$-orthonormalized functions.
   \begin{remark}
     In practice, the algorithm is halted with a stopping criterion such as an error threshold or a maximum number of basis functions to generate.
 \end{remark}
 \item Then, we solve the following eigenvalue problem:
    \begin{numcases}
      \strut \textrm{Find } \Phi^{h} \in X_h^{N}, \textrm{ and } \lambda \in \mathbb{R} \textrm{ such that: }    \nonumber\\
        \forall \vv \in X_h^{N}, \int_{\Omega} \nabla \Phi^{h} \cdot \nabla \vv \ d\xx= \lambda \int_{\Omega} \Phi^{h} \cdot \vv \ d\xx, \label{orthohuhu} 
    \end{numcases} 
   We get an increasing sequence of eigenvalues $\lambda_i$, and orthogonal eigenfunctions $(\Phi^{h}_{i})_{i=1,\cdots,N}$, which do not depend on time, orthonormalized in $L^2(\Omega)$ and orthogonalized in $H^1(\Omega)$. Note that with the Gram-Schmidt procedure of the classical greedy {}{A}lgorithm \ref{algogreedy}, we only obtain an $L^2$-orthonormalized RB. 
\item We remark that, for any parameter  $\mu_k, \ k=1,\dots,N_{\mu}$, the classical NIRB approximation differs from the HF $u_h(\mu_k)$ computed in the offline stage (see also \cite{paraboliccontext}). Thus, to improve NIRB accuracy, we use a "rectification post-processing". We construct rectification matrices denoted $\mathbf{R}^n$ for each fine time step $n=0,\dots,N_T$. They are built from the HF snapshots $u_h^n$, $n=0,\dots,N_{T}$ and {}{ coarse snapshots} generated with the higher order scheme (e.g. \eqref{varpara2discCN}), {}{interpolated in time at the fine time steps by \eqref{parabola},} and whose parameters are the same as for the fine snapshots. 
More precisely, for all $n = 0,\dots, N_T,$ for all $i=1,\cdots,N,\  \textrm{and}  \ $for all $\mu_k \in  \mathcal{G}_{train}$, we define the \textit{coarse coefficients} as 
\begin{equation}
    \quad  A_{k,i}^n=\int_{\Omega} \widetilde{u_H}^n(\mu_k) \cdot \Phi_i^h\ \textrm{d}\xx, \label{Aj}
\end{equation}
{}{where $\widetilde{u_H}^n$ stands for an interpolation of second order in time of the coarse snapshots, }and the \textit{fine coefficients} as
\begin{equation}
    \quad B_{k,i}^n=\int_{\Omega} \uu_h^n(\mu_k) \cdot \Phi_i^h\ \textrm{d}\xx.     \label{Bj}
\end{equation}
Then, we compute the vectors
    \begin{equation}
              \mathbf{R}_{u,i}^n=((\mathbf{A}^n)^T\mathbf{A}^n+\delta \mathbf{I}_{N})^{-1}(\mathbf{A}^n)^T \mathbf{B}_i^n, \quad  i=1, \cdots,N,
              \label{rectiff1}
    \end{equation}
    where $\mathbf{I}_{N}$ refers to the identity matrix and $\delta$ is a regularization parameter.
 \begin{remark}
 Note that since every time step has its own rectification matrix, the matrix $\mathbf{A}^n$ is a ``flat'' rectangular matrix 
($N_{train}\leq N$), and thus the parameter $\delta$ is required for the inversion of $(\mathbf{A}^n)^T\mathbf{A}^n$. \\
  
   \end{remark}

   {}{To sum up this offline procedure, we resort to Algorithm \ref{offline} where the two last steps are concerned with the rectification post-treatment. }
   \begin{algorithm}
     \caption{{}{Offline rectified NIRB algorithm }}  \label{offline}
     {}{  \textbf{Input:} $tol$, $\mathcal{G}_{train}$  \\
    \textbf{Output:} Reduced basis $\{\Phi_1^h,\cdots,\Phi_{N}^h\}$. \\
    }
    \begin{algorithmic}
    \State    {}{ $\triangleright$ Construct  $\{\uu_h^n(\mu_1),\cdots,\uu_h^n(\mu_{N_{train}}) \textrm{ with } \mu_i  \in \mathcal{G}_{train}, \ n=0,\dots, N_T \}.$}
      \State {}{$\triangleright$ Find a RB $\{\Phi_1^h,\cdots,\Phi_{N}^h\}$ of $X_h^N$ via Algorithm \ref{algogreedy}.
      \State $\triangleright$ Solve \eqref{orthohuhu}.
      \State $\triangleright$ Generate $\{\widetilde{u_H}^n(\mu_1),\cdots ,\widetilde{u_H}^n(\mu_{N_{train}}) \textrm{ with } \mu_i~\in~\mathcal{G}_{train}, \ n=0,\dots, N_T \}$ by computing the coarse snapshots and interpolating them via \eqref{parabola}.
      \State $\triangleright$ Construct the rectification matrices \eqref{rectiff1}. }
  \end{algorithmic}
  \end{algorithm}
 \end{enumerate}
\item \textbf{``Online step''}\\
  The online part of the algorithm is much faster than a {}{fine }HF evaluation. 
  \begin{enumerate}
    \setcounter{enumi}{0}
  \item 
   For a new parameter $\mu \in \mathcal{G}$ we are interested in, we solve the problem with the higher order scheme (e.g. \eqref{varpara2discCN}) on the coarse mesh $\mathcal{T}_H$ at each time step $m=~0,\dots, M_T$. 
  \item   We quadratically interpolate in time the coarse solution on the fine time grid with \eqref{parabola}. 
    \item  Then, we linearly interpolate in space $\widetilde{u_H}^n(\mu)$ on the fine mesh in order to compute the $L^2$-inner product with the RB functions. Finally, the approximation used in the two-grid method is
      \begin{equation}
        \label{NIRBapproximation0}
     \textrm{For }n=0,\dots, N_T, \quad   u_{Hh}^{N,n}(\mu):= \overset{N}{\underset{i=1}{\sum}}(\widetilde{u_H}^n(\mu),\Phi^{h}_{i})\ \Phi^{h}_{i},
      \end{equation}
      and with the rectification post-treatment step, it becomes
         \begin{equation}
        \label{NIRBapproximation0rect}
       \mathbf{R}_u^n[u_{Hh}^{N}](\mu):= \overset{N}{\underset{i,j=1}{\sum}}R_{u,ij}^n\ (\widetilde{u_H}^n(\mu),\Phi^{h}_{j})\ \Phi^{h}_{i},
         \end{equation}
         where $(\mathbf{R}^{n}_u)_{ij}=R_{u,ij}^n$ is the rectification matrix at time $t^n$, given by \eqref{rectiff1}.
\end{enumerate}

  \begin{algorithm}
     \caption{{}{Online (classical or rectified) NIRB algorithm }  }\label{online}
  {}{     \textbf{Input:}  Reduced basis $\{\Phi_1^h,\cdots,\Phi_{N}^h\}, \ \textrm{and a parameter of interest }\mu \in \mathcal{G}$ ; Rectification matrices in case of the rectified version\\
    \textbf{Output:} NIRB approximation.} \\
    \begin{algorithmic}
    \State {}{$\triangleright$ Compute the coarse solution associated to $\mu$.
      \State  $\triangleright$ Interpolate it via \eqref{parabola2}.
  \State $\triangleright$ Construct the NIRB approximation via \eqref{NIRBapproximation0} or \eqref{NIRBapproximation0rect}.}
  \end{algorithmic}
  \end{algorithm}
 \end{itemize}
In \cite{paraboliccontext}, we have proven the following estimate on the heat equation
    \begin{equation}
\textrm{ for } n = 0,\dots, N_T, \ \Big|u(t^n)(\mu)-u_{Hh}^{N,n}(\mu)\Big|_{H^1(\Omega)}\leq \varepsilon(N) + C_1(\mu) h +C_2(N) H^2 +  C_3(\mu) \Delta t_F +C_4(N) \Delta t_G^2,    \label{EllipticEst0}
\end{equation}
    where $C_1, C_2, C_3$ and $C_4$ are constants independent of $h$ and $H$, $\Delta t_F$ and $\Delta t_G$.
     The term $\varepsilon(N)$ depends on a proper choice of the RB space as a surrogate for the best approximation space associated to the Kolmogorov $N$-width. 
    It decreases when $N$ increases and it is linked to the error between the fine solution and its projection on the reduced space $X_h^N$, given by
    \begin{equation}
      \label{truerror}
  \Big|u_h^n(\mu) - \overset{N}{\underset{i=1}{\sum}}(u_h^n(\mu),\Phi_i^h)\ \Phi_i^h\Big|_{H^1(\Omega)}.
    \end{equation}
   On the other hand, the constants $C_2(N)$ and $C_4(N)$ in \eqref{EllipticEst0} increase with $N$. Thus, a trade-off needs to be done between increasing $N$ to obtain a more accurate manifold, and keeping these constants as low as possible. {}{Details on the behavior of these constants are provided in \cite{paraboliccontext}}.\\
If $H$ is such as $H^2 \sim h$ and $\Delta t_G^2\sim \Delta t_F$ with $C_2(N)$ and $C_4(N)$ not too large, the estimate \eqref{EllipticEst0} entails an error estimate in $\varepsilon(N) + \mathcal{O}(h + \Delta t_F)$. Therefore, if $\varepsilon(N)$ is small enough, we recover an optimal error estimate in $L^{\infty}(0,T;H^1_0(\Omega))$. 
Our aim is to extend this error estimate \eqref{EllipticEst0} to the context of sensitivity analysis.

\section{NIRB error estimate on the sensitivities}
\label{NIRBproof}

{}{Let $u(\mu)$ be the exact solution of $\mathcal{P}$ (e.g. \eqref{heatEq2} for the heat equation) for a parameter  $\mu=(\mu_1,\dots,\mu_P) \in \mathcal{G}$ and $\Psi(\mu_p)$ its \sens with respect to the parameter $\mu_p, p=1,\dots, P$. We define the training set $\mathcal{G}_{p,train}= \underset{i \in \{1,\dots,N_{p,train}\}}{\cup} \widetilde{\mu}_i$, from a set of training parameters $(\widetilde{\mu}_i)_{i=1,\cdots,N_{p,train}}.$
The RB function are generated via the offline steps described above (see Algorithms \ref{algogreedy} and \ref{offline}), but from fine HF fully-discretized sensitivities  $\{\Psi_{p,h}^n(\widetilde{\mu}_i)\}_{i \in \{1,\dots N_{\mu,p}\},\ n=\{0,\dots,N_T\}}$ ($N_{\mu,p}\leq N_{p,train}$) instead of the state solutions $\{u_h^n(\mu_i)\}_{i=\{1, \dots, N_{\mu}\}, n=\{0,\dots,N_T\}}$ (e.g. with  \eqref{eq:directheatsensitivity} in case of the heat equation).
At the end of this procedure, we obtain  $N_p$ $L^2$ and $H^1$ orthogonal RB (time-independent) functions, denoted $(\zeta^{h}_{p,i})_{i=1,\dots,N_p}$, of the reduced spaces  $X_{p,h}^{N_p}:= \textrm{Span}\{\zeta_{p,1}^h,\dots,\zeta_{p,N_p}^h\}$ for $p=1,\dots,P$, as well as rectification matrices, that we denote $\mathbf{R}_{\Psi}^{p,n}$, associated to the {}{sensitivity} problem (\eqref{eq:directheatsensitivity} for the heat equation), for each $p \in \{1,\dots, P\}$ and each fine time step $n \in \{0,\dots, N_T\}$, such that
    \begin{equation}
              \mathbf{R}_{\Psi,i}^{p,n}=((\mathbf{A}^{p,n})^T\mathbf{A}^{p,n}+\delta_p \mathbf{I}_{N_p})^{-1}(\mathbf{A}^{p,n})^T \mathbf{B}_{i}^{p,n}, \quad  i=1, \cdots,N_p,
              \label{rectiff2}
    \end{equation}
    where
    \begin{align}
      & \forall i=1,\cdots,N_p,\quad  \textrm{ and }  \quad \forall \widetilde{\mu}_k \in  \mathcal{G}_{p},\nonumber \\
      & \quad  A_{k,i}^{p,n}=\int_{\Omega} \widetilde{\Psi_{p,H}}^n(\widetilde{\mu}_k) \cdot \zeta_{p,i}^h\ \textrm{d}\xx, \label{Ajj} \\
      & \quad B_{k,i}^{p,n}=\int_{\Omega} \Psi_{p,h}^n(\widetilde{\mu}_k) \cdot \zeta_{p,i}^h\ \textrm{d}\xx,     \label{Bjj}
\end{align}
 and where we use \eqref{parabola2} to compute $\widetilde{\Psi_{p,H}}^n(\widetilde{\mu}_k)$.
}

{}{
The online part of the algorithm is much faster than {}{P+1} fine HF evaluations and follows the NIRB two-grid online part, described above in the context of parabolic state equations \ref{NIRBParabolic} (see Algorithm \ref{online}).
 The approximation on the sensitivity obtained from the classical two-grid method is
      \begin{equation}
        \label{NIRBapproximation00}
     \textrm{For }n=0,\dots, N_T, \quad   \Psi_{p,Hh}^{N_p,n}(\mu):= \overset{N_p}{\underset{i=1}{\sum}}(\widetilde{\Psi_{p,H}}^n(\mu),\zeta^{h}_{p,i})\ \zeta^{h}_{p,i},
      \end{equation}
      and with the rectification post-treatment step, it becomes
         \begin{equation}
        \label{NIRBapproximation0rect0}
      \mathbf{R}_{\Psi}^{p,n}[\Psi_{p,Hh}^{N_p}](\mu):= \overset{N_p}{\underset{i,j=1}{\sum}}R_{ij}^{p,n}\ (\widetilde{\Psi_{p,H}}^n(\mu),\zeta^{h}_{p,j})\ \zeta^{h}_{p,i},
         \end{equation}
         where $\mathbf{R}_{\Psi}^{p,n}$ is the rectification matrix at time $t^n$, given by \eqref{rectiff2}.
}
\begin{paragraph}{Main result}
Our main theoretical result is the following theorem, which applies to the direct sensitivity formulation.
\begin{theorem}{(NIRB error estimate for the sensitivities.)}
\label{EllipticEst}
Let $A(\mu)=\mu \ I_d$, with $\mu \in \mathbb{R}_*^+$, and let us consider the problem \ref{heatEq2} with its exact solution $u(\xx, t; \mu)$, and the full discretized solution $u_h^n(\xx;\mu)$ to the problem \ref{varpara2disc}. 
Let $\Psi(\xx, t ; \mu)$ and $\Psi_h^n(\xx;\mu)$ {}{be} the corresponding \sensfct, given by \eqref{eq:2chainrule} and \eqref{eq:directheatsensitivity}.
Let $(\zeta^h_{i})_{i=1,\dots,N_1}$ be the $L^2$-orthonormalized and $H^1$-orthogonalized RB generated with the greedy {}{A}lgorithm \ref{algogreedy}.
Let us consider the NIRB approximation defined by \eqref{NIRBapproximation00} i.e.
 \begin{equation*}
     \textrm{for }n=0,\dots, N_T, \quad   \Psi_{Hh}^{N,n}(\mu):= \overset{N_1}{\underset{i=1}{\sum}}(\widetilde{\Psi_H}^n(\mu),\zeta^{h}_{i})\ \zeta^{h}_{i},
 \end{equation*}
 where $\widetilde{\Psi_H}^n(\mu)$ is given by \eqref{parabola2}.
Then, the following estimate holds
\begin{equation}
\forall n = 0,\dots, N_T, \ \Big|\Psi(t^n)(\mu)-\Psi_{Hh}^{N,n}(\mu)\Big|_{H^1(\Omega)}\leq \varepsilon(N) + C_1(\mu) h +C_2(\mu,N) H^2 +  C_3(\mu) \Delta t_F +C_4(\mu,N) \Delta t_G^2,  \label{estest2}
\end{equation}
where $C_1, C_2, C_3$ and $C_4$ are constants independent of $h$ and $H$, $\Delta t_F$ and $\Delta t_G$. The term $\varepsilon$ depends on the Kolmogorov $N$-width and {}{measurements} the error given by \eqref{truerror}. 
  \end{theorem}
If $H$ is such as $H^2 \sim h$, $\Delta t_G^2\sim \Delta t_F$, and $C_2(\mu,N)$ and $C_4(\mu,N)$ not too large, it results in an error estimate in $\varepsilon(N) + \mathcal{O}(h + \Delta t_F)$.
Theorem \ref{EllipticEst} then states that we recover optimal error estimates in $L^{\infty}(0,T;H^1(\Omega))$ if $\varepsilon(N)$ is small enough. 
\end{paragraph}
We now go on with the proof of Theorem \ref{EllipticEst}.
\begin{proof}
  The NIRB approximation at time step $n=0,\dots,N_T$, for a new parameter $\mu \in \mathcal{G}$ is defined by \eqref{NIRBapproximation00}.
Thus, the triangle inequality gives
\begin{align}
\ \Big|\Psi(t^n)(\mu)-\Psi_{Hh}^{N,n}(\mu)\Big|_{H^1(\Omega)}&\leq \Big|\Psi(t^n)(\mu)- \Psi_{h}^n(\mu)\Big|_{H^1(\Omega)}+\Big|\Psi_h^n(\mu)-\Psi_{hh}^{N,n}(\mu)\Big|_{H^1(\Omega)}+ \Big|\Psi_{hh}^{N,n}(\mu)-\Psi_{Hh}^{N,n}(\mu)\Big|_{H^1(\Omega)} \nonumber \\
&=:T_1+T_2+T_3, 
\label{triangleinequalitytime}
\end{align}
where $\Psi_{hh}^{N_1,n}(\mu)=\overset{N_1}{\underset{i=1}{\sum}}  (\Psi_h^n(\mu),\zeta^h_i)\ \zeta^h_i$.\\
\begin{itemize}
    \item 
    The first term $T_1$ may be estimated using the inequality given by Theorem \ref{corollaryPsiiH1disc},
such that 
\begin{equation}
    \Big|\Psi(t^n)(\mu)-\Psi_h^n(\mu)\Big|_{H^1(\Omega)}\leq C(\mu)\ (h + \Delta t_F).
    \label{cea2}
\end{equation}

\item We then denote by $\mathcal{S}_h'=\{ \Psi_h^n(\mu,t),\mu~\in~\mathcal{G},\ n=0,\dots N_T \}$ the set of all the \sensfct. For our model problem, this manifold has a low complexity. It means that for an accuracy $\varepsilon=\varepsilon(N)$ related to the Kolmogorov $N$-width of the manifold $\mathcal{S}_h'$, for any $\mu \in \mathcal{G}$, and any $n \in 0,\dots,N_T$, $T_2$ is bounded by $\varepsilon$ which depends on the Kolmogorov $N$-width.
\begin{equation}
  T_2=\Big|\Psi_h^n(\mu)-\overset{N_1}{\underset{i=1}{\sum}}(\Psi_h^n(\mu),\zeta^h_i)\ \zeta^h_i\Big|_{H^1(\Omega)} \leq \varepsilon(N).\label{kolmoFEMnew}
  \end{equation}
\item Since $(\zeta_i^h)_{i=1,\dots,N_1}$ is a family of $L^2$ and $H^1$ orthogonalized RB functions (see \cite{paraboliccontext} for only $L^2$ orthonormalized RB functions)
    \begin{equation}
\label{term3ortho}
    \Big|\Psi_{hh}^{N,n}-\Psi_{Hh}^{N,n}\Big|_{H^1(\Omega)}^2 = \overset{N_1}{\underset{i=1}{\sum}} |(\Psi_h^n(\mu)-\widetilde{\Psi_H}^n(\mu),\zeta^{h}_{i})|^2 \big|\zeta^{h}_{i}\big|_{H^1(\Omega)}^2,
    \end{equation}
    where $\widetilde{\Psi_H}^n(\mu)$ is the quadratic interpolation of the coarse snapshots on time $t^n$,\ $\forall n= 0,\dots, N_T$, defined by \eqref{parabola2}.
From the RB orthonormalization in $L_2$, the equation \eqref{orthohuhu} yields
\begin{equation}
  \label{orthoplus}
    \big|\zeta^{h}_{i}\big|_{H^1}^2 =\lambda_i \norm{\zeta^{h}_{i}}_{L^2(\Omega)}^2
    =\lambda_i \leq \underset{i=1,\cdots,N}{\max \ }\lambda_i=\lambda_N,
\end{equation}
such that the equation  \eqref{term3ortho} leads to
\begin{equation}
  \label{ici}
     \Big|\Psi_{hh}^{N,n}-\Psi_{Hh}^{N,n}\Big|_{H^1(\Omega)}^2    \leq    C \lambda_N\norm{\Psi_h^n(\mu)-\widetilde{\Psi_H}^n(\mu)}_{L^2(\Omega)}^2.
\end{equation}

Now by definition of $\widetilde{\Psi_H}^n(\mu)$ and by Corollary \ref{corollaryCorollaryPsiiCrankNicolson} and Theorem \ref{corollaryPsii}, for $t^n \in I_m,$\\
\begin{equation}
  \label{Aubinlike}
  \norm{\Psi_h^n(\mu)-\widetilde{\Psi_H}^n(\mu)}_{L^2(\Omega)} \leq  C(\mu)( H^2+ \Delta t_G^2+h^2 + \Delta t_F ), 
\end{equation}
and  we end up for equation \eqref{ici} with
\begin{equation}
    \Big|\Psi_{hh}^{N,n}-\Psi_{Hh}^{N,n}\Big|_{H^1(\Omega)} \leq C(\mu)\sqrt{\lambda_N} (H^2 +  \Delta t_G^2 + h^2 + \Delta t_F),
    \label{T3termParaboo}
\end{equation}
where $C(\mu)$ does not depend on $N$. Combining these estimates \eqref{cea2}, \eqref{kolmoFEMnew} and \eqref{T3termParaboo} concludes the proof and yields the estimate \eqref{estest2}.
\end{itemize}
\end{proof}

\section{New NIRB algorithms applied to sensitivity analysis}
\label{NIRBalgoS}
In this section, we propose several adaptations of the NIRB two-grid algorithm to the context of sensitivity analysis with the direct and adjoint formulations. The differentiability of the state $u$ with respect to parameters $\mu\in \mathcal{G}$ is assumed in what follows.
\subsection{NIRB two-grid GP algorithm for the direct problem.}
\label{NIRBalgosensdirect}

In this section, we propose a NIRB two-grid algorithm with a new post-treatment which reduces the online computational time.

\label{rectificationplus}
The main drawback of the algorithm described in the previous section is that it requires $1+P$ coarse systems in the online part (see {}{S}ection \ref{NIRBParabolic} in the context of state parabolic equations). The online portion of the new algorithm described below  requires the resolution of one coarse problem only, regardless of the number of parameters of interest. The idea is to construct a learning post-treatment between the coarse and the fine coefficients as in \cite{grosjean:hal-03588508} with different RB but here through a Gaussian Process  Regression (GPR) \cite{GUO2018807,nguyen2016gaussian} instead of a deterministic procedure. We adapt the Gaussian Processes (GP) employed in \cite{guo2019data} by using the coarse coefficients as inputs of our GP instead of the time steps and parameters. We have tested different kernels for the covariance function and we consider time-independent RB functions. This new non-intrusive algorithm is simple to implement and it may yield very good performance in both accuracy and efficiency for time-dependent problems applied to sensitivity analysis, as shown in {}{S}ection \ref{results}.

We refer to the following offline/online decomposition:
\begin{itemize}
\item \textbf{``Offline part''}\\
  \begin{enumerate}
  \item From a parameter training set $\mathcal{G}_{T,train}:=\mathcal{G}_{train} \cup \mathcal{G}_{p,train}$ (for $p=~1,\dots,P$), we seek the RB parameters in $ \mathcal{G}_T:= \mathcal{G}_N \cup \mathcal{G}_p$ for the state and sensitivity solutions through a greedy algorithm, as in {}{A}lgorithm \ref{algogreedy}, that allows us to generate the modes of the state and sensitivity reduce spaces, $X_h^{N_{\mu},T}$ and $X_{p,h}^{N_{\mu,T}} $ respectively.
We denote $N_{\mu,T}$ the cardinality of $\mathcal{G}_{T}$.     
At the end of this part, we obtain $N_{\mu,T}$ RB functions for the problem solutions denoted $(\Phi^h_i)_{i=1,\dots,N_{\mu,T}}$ and $N_{\mu,T}$ RB functions for the sensitivities, denoted $(\zeta^{h}_{p,i})_{i=1,\dots,N_{\mu,T}}$ for each $p=1,\dots,P$. 
\item \textbf{Gaussian Process  Regression (GPR).}\begin{paragraph}{} We then use the fact that the sensitivities are directly derived from the initial solutions, and thus, we consider a Gaussian metamodel to recover the fine coefficients of the sensitivities \eqref{Bjj} from the coarse coefficients of the problem solutions \eqref{Ajj}. 
 In addition to the fine sensitivity coefficients already computed for the RB generation, we compute the coarse coefficients for the problem solutions. Thus, we define
 \begin{align}
      &   \forall \widetilde{\mu}_k \in  \mathcal{G}_{T},\ \forall i=1,\cdots,N_{\mu,T}, \ \forall n=0,\cdots,N_{T} , \nonumber \\
      & \quad \quad \quad  A_{i}^{n}(\widetilde{\mu}_k)=\int_{\Omega} \widetilde{u_{H}}^n(\widetilde{\mu}_k) \cdot \Phi_{i}^h\ \textrm{d}\xx,\label{Ajj} \\
      & \quad  \quad \quad B_{i}^{p,n}(\widetilde{\mu}_k)=\int_{\Omega} \Psi_{p,h}^n(\widetilde{\mu}_k) \cdot \zeta_{p,i}^h\ \textrm{d}\xx,   \label{Bjj}
\end{align} 
and use them as a set of inputs-outputs for our GPR model,\  \begin{equation}
\label{DDeq}
    \mathcal{D}=\{(\mathbf{A}_k,\mathbf{B}_k): \ k=1,\dots, N_{T,train}\},
\end{equation}
with $\mathbf{A}_k=(A_{i}^n(\widetilde{\mu}_k))_{n=0,\dots,N_T, \ i=1,\dots,N_{\mu,T}}$ and  $\mathbf{B}_k=(B_{i}^n(\widetilde{\mu}_k))_{n=0,\dots,N_T, \ i=1,\dots,N_{\mu,T}}$ for \ $k=1,\dots, N_{T,train}$. We denote $\mathbf{A}=\{\mathbf{A}_1,\dots,\mathbf{A}_{N_{T,train}}\}$ the training inputs and by $\mathbf{B}=\{\mathbf{B}_1,\dots,\mathbf{B}_{N_{T,train}}\}$ the associated outputs.\\

Here, we deal with $3$ kinds of indices:
\begin{itemize}
    \item the number of modes: $N_{\mu,T}$,
    \item the cardinality of the training set: $N_{T,train}$,
    \item and the number of time steps which is $N_T+1.$
\end{itemize}The observed input-output pairs in $\mathcal{D}$ are assumed to follow some unknown regression function 
\begin{equation*}
f : \mathbb{R}^{N_{\mu,T}\times (N_T+1)} \to~\mathbb{R}^{N_{\mu,T}\times (N_T+1)}, \ f(\mathbf{A}_k) = \mathbf{B}_k, \ k=1,\dots,N_{T,train}.
\end{equation*} 
From a Bayesian perspective, we can define a prior and posterior GP on the regression: 
\begin{itemize}
\item the prior GP reflects our beliefs about the metamodel before seeing any training data, and is solely defined by a mean and a covariance function;
\item the posterior GP conditions the prior on the training data, i.e. includes the knowledge from the data $\mathcal{D}$. 
\end{itemize}

{}{The prior GP is defined by a mean and a covariance function such that
\begin{equation*}
\textrm{for } (\xx,\xx') \in \mathbb{R}^{N_{\mu,T}\times (N_T+1)} \times  \mathbb{R}^{N_{\mu,T}\times (N_T+1)}, f(\xx) \sim GP(0, \kappa(\xx,\xx')), \mathbf{y}=f(\xx), 
\end{equation*}
}
where the covariance function $\kappa$ is a positive definite kernel (see \cite{duvenaud2014automatic} for an overview on different kernels). There are many different options for the covariance functions. {}{For our model }problem, we use the standard squared exponential covariance function ({}{which is a special instance of} radial basis functions) 
\begin{equation}
\label{kappa}
    \textrm{for } (\xx,\xx')\in \mathbb{R}^{N_{\mu,T}\times (N_T+1)} \times \mathbb{R}^{N_{\mu,T}\times (N_T+1)}, \ \kappa(\mathbf{x},\mathbf{x}') = \sigma_f^2 \ \textrm{exp} \big(- \frac{1}{2l^2} \norm{\xx - \xx'}\big),
\end{equation}
with $\norm{\cdot}$ being the usual  Euclidean norm and where two hyperparameters are employed: the standard deviation parameter $ \sigma_f$ and the correlated length-scale $l$. We assume that the prior mean function is zero which is common practice and does not limit the GP model. \\


The goal of GPR is then to use the training data $\mathcal{D}$ \eqref{DDeq} to make predictions when given new inputs.\\
In our case, those new points (where we predict the output) will be the coarse coefficients given by \eqref{Ajj} for a new parameter $\mu \in \mathcal{G}$. \\
To incorporate our knowledge from the training data, we condition the prior GP on our set of training data points $\mathcal{D}$ \eqref{DDeq}. The posterior distribution of the output $f^{*}$ for a new input $\mathbf{A}^{*} \in \mathbb{R}^{N_{\mu,T}\times (N_T+1)}$ is then given by \cite{wirthl2022global}
\begin{align}
\label{GPReq}
   & f^*|\mathbf{A}^{*},\mathbf{A}, \mathbf{B} \sim \mathcal{N}(m^*(\mathbf{A}^{*}),C^*(\mathbf{A}^{*},\mathbf{A}^{*})),\notag \\
   & \textrm{ with } m^*(\textbf{A}^{*})=K^{*T} K_y^{-1}\mathbf{B}, \notag \\
   & \textrm{ and } C^*(\mathbf{A}^{*},\mathbf{A}^{*})= K^{**}-K^{*T} K_y^{-1} K^*,
\end{align}
where $K_y= \kappa(\mathbf{A},\mathbf{A})+ \sigma_y^2 \mathbf{I}_{N_{T,train}}$ ($\mathbf{I}_{N_{T,train}}$ being the $N_{T,train}$-dimensional unit matrix), $K^*=\kappa(\mathbf{A}^*,\mathbf{A})$ and $K^{**}=\kappa(\mathbf{A}^*,\mathbf{A}^*)$.

The hyperparameters $\theta = (\sigma_f, l)$ in \eqref{kappa} are optimised by maximising the log marginal likelihood using a gradient-based optimiser. The log marginal likelihood is given by
\begin{equation*}
    \theta_{opt} \in \underset{\theta}{\arg \max} \log p(\mathbf{B}|\mathbf{A}) = \underset{\theta}{\arg \max} \{ -\frac{1}{2} \mathbf{B}^T K_y^{-1} (\theta) \mathbf{B} - \frac{1}{2} \log|K_y(\theta)| - \frac{N_{T,train}}{2} \log (2\pi) \},
\end{equation*}
where $p(\mathbf{B}|\mathbf{A})$ is the conditional density function of $\mathbf{B}$ given $\mathbf{A}$, also considered as the log marginal likelihood. Note that computing the inverse of $K_y$ may be expensive computationally, i.e. on the order $\mathcal{O}(N_{T,train}^3)$, increasing the offline computational cost when the number of training samples increases. Yet, we get rid of the time complexity since we employ a "global" GP on time and the covariance function defined by \eqref{kappa}.

\end{paragraph}
 \end{enumerate}
\item \textbf{``Online step''}\\
  \begin{enumerate}
    \setcounter{enumi}{0}
  \item We solve the problem \eqref{heatEq2} on the coarse mesh $\mathcal{T}_H$ for a new parameter $\mu \in \mathcal{G}$ at each time step $m=~0,\dots, M_T$ using \eqref{varpara2discCN}. 
  \item   We interpolate quadratically in time the coarse solution $u_{H}^m$ on the fine time grid with \eqref{parabola}.
    \item  Then, we linearly interpolate $\widetilde{u_{H}}^n(\mu)$ on the fine mesh in order to compute the $L^2$-inner product with the basis functions. We derive the new coarse coefficients with 
    \begin{equation}
     A^{n}_j(\mu)  = (\widetilde{u_{H}}^n(\mu),\Phi^{h}_{j}) \textrm{ for } j=1,\dots, N_{\mu,T} \textrm{ and } n=0,\dots, N_T.
    \end{equation}
    We denote $\mathbf{A}^*(\mu)$ these coefficients. Then, the new sensitivity fine coefficients $\mathbf{B}^{GP}$ are approximated following \eqref{GPReq}, thus
     \begin{equation*}
     \mathbf{B}^{GP} (\mathbf{A}^*(\mu)) \sim GP(m^*(\mathbf{A}^*(\mu)),C^*(\mathbf{A}^*(\mu),\mathbf{A}^*(\mu))),
    \end{equation*}
    with $m^*$ and $C^*$ defined by \eqref{GPReq}.
    Finally, the new NIRB approximation is obtained by
         \begin{equation}
        \label{NIRBapproximation0rect00}
      \Psi_{p,Hh}^{GP,n}(\mu):= \overset{N_{\mu,T}}{\underset{i=1}{\sum}} \   [\mathbf{B}^{GP}(\mathbf{A}^*(\mu))]^n_i \ \zeta^{h}_{p,i}.
         \end{equation}
\end{enumerate}
 \end{itemize}
 Now, we proceed with the NIRB two-grid algorithm adapted to the adjoint formulation.
\subsection{On the adjoint formulation.}
\label{NIRBalgosensadjoint}
The adjoint formulation requires some modifications of the NIRB two-grid algorithm presented in {}{S}ection \ref{NIRBParabolic}. Since in \eqref{Dmup1}, for all $n\in \{0,\dots,N_T\},$ the fine solution $u_h^n(\mu)$ is required to obtain the sensitivities on $\mathcal{F}$, it follows that here we have to compute two reductions: one for the initial solution $u$ and one for the adjoint $\chi$. 
So let $u(\mu)$ be the exact solution of problem \eqref{heatEq2} for a parameter  $\mu \in \mathcal{G}$ and $\chi(\mu)$ its adjoint given by \eqref{eq:chicontinuous}.
In this setting, we use the following offline/online decomposition for the NIRB procedure:
\begin{itemize}
\item \textbf{``Offline part''}\\
  \begin{enumerate}
  \item During the offline stage, we first construct the reduced space $X_h^N$ and the RB function $(\Phi_1^h,\dots,\Phi_N^h)$ with the steps 1-2 of {}{S}ection \ref{NIRBParabolic}.
\item  Then, we use the same steps 1-2 but with the adjoint problem on the fine mesh, as  \eqref{adjointdisceulerweak} in case of the heat equation. We denote by $X_{1}^{N_1}$ the reduced space. \\
Thus, for a set of training parameters $(\widetilde{\mu}_i)_{i=1,\cdots,N_{1,train}},$ we define $\mathcal{G}_{1,train}= \underset{i \in \{1,\dots,N_{1,train}\}}{\cup} \widetilde{\mu}_i$. Then, through a greedy procedure  \ref{algogreedy}, we adequately choose the parameters of the RB. During this procedure, we compute fine fully-discretized solutions $\{\chi_{h}^n(\widetilde{\mu}_i)\}_{i \in \{1,\dots N_{\mu,1}\},\ n=\{0,\dots,N_T\}}$ ($N_{\mu,1}\leq N_{1,train}$) with the HF solver (e.g. by solving either \eqref{adjointdisceulerweak} or the associated problem where $u_h^n$ is replaced by its NIRB rectified version $\mathbf{R}^n_u[u_{Hh}^{N,n}]$ obtained from the algorithm of {}{S}ection \ref{NIRBParabolic}).
   In analogy with {}{S}ection \ref{NIRBParabolic}, a few time steps may be selected for each parameter of the RB, and thus, we obtain $N_1$ $L^2$ orthogonal RB (time-independent) functions, denoted $(\xi^{h}_{i})_{i=1,\dots,N_1}$, and the reduced space $X_{h}^{N_1}:= \textrm{Span}\{\xi_{1}^h,\dots,\xi_{N_1}^h\}$.
 \item Then, we solve the eigenvalue problem \eqref{orthohuhu} on $X_{h}^{N_1}$:
    \begin{numcases}
      \strut \textrm{Find } \xi^{h} \in X_{h}^{N_1}, \textrm{ and } \lambda \in \mathbb{R} \textrm{ such that: }    \nonumber\\
        \forall \vv \in X_{h}^{N_1}, \int_{\Omega} \nabla  \xi^{h} \cdot \nabla \vv \ d\xx= \lambda \int_{\Omega}  \xi^{h} \cdot \vv \ d\xx. \label{orthohuhup}
    \end{numcases} 
  We get an increasing sequence of eigenvalues $\lambda_i$, and eigenfunctions $(\xi^h_{i})_{i=1,\cdots,N_1}$, orthonormalized in $L^2(\Omega)$ and orthogonalized in $H^1(\Omega)$. 
\item As in the offline step 3 from {}{S}ection \ref{NIRBParabolic},  the NIRB approximation is enhanced with a rectification post-processing. Thus, we introduce a rectification matrix, denoted $\mathbf{R}_{\chi}^{n}$ for each fine time step $n \in \{0,\dots, N_T\}$, such that 
    \begin{equation}
              \mathbf{R}_{\chi,i}^{n}=((\mathbf{A}^{n})^T\mathbf{A}^{n}+\delta \mathbf{I}_{N_1})^{-1}(\mathbf{A}^{n})^T \mathbf{B}_{i}^{n}, \quad  i=1, \cdots,N_1,
              \label{rectiff5}
    \end{equation}
    where
    \begin{align}
      & \forall i=1,\cdots,N_1,\quad  \textrm{ and }  \quad \forall \widetilde{\mu}_k \in  \mathcal{G}_{p},\nonumber \\
      & \quad  A_{k,i}^{n}=\int_{\Omega} \widetilde{\chi_{H}}^n(\widetilde{\mu}_k) \cdot \xi_{i}^h\ \textrm{d}\xx, \label{Ajjj} \\
      & \quad B_{k,i}^{n}=\int_{\Omega} \chi_{h}^n(\widetilde{\mu}_k) \cdot \xi_{i}^h\ \textrm{d}\xx,     \label{Bjjj}
\end{align}
 and where $\mathbf{I}_{N_1}$ refers to the identity matrix and $\delta_p$ is a regularization term required for the inversion of $(\mathbf{A}^{n})^T\mathbf{A}^{n}$ (note that we used \eqref{parabola2} for $\widetilde{\chi_{H}}^n(\widetilde{\mu}_k)$).
 \end{enumerate}
\item \textbf{``Online part''}\\
  \begin{enumerate}
    \setcounter{enumi}{0}
  \item We first solve the forward problem (e.g. \eqref{heatEq2}) on the coarse mesh $\mathcal{T}_H$ for a new parameter $\mu \in \mathcal{G}$ at each time step $m=0,\dots, M_T$ using a second order time scheme for the discretization in time (e.g. \eqref{varpara2discCN}). 
   \item Then, we solve the coarse associated adjoint problem (e.g. \eqref{adjointdisceulerweakCN}) with the same parameter $\mu$,  at each time step $m=0,\dots, M_T$.
  \item   We interpolate quadratically in time the coarse solution $\chi_{H}^m$ on the fine time grid with \eqref{parabola2}.
    \item  Then, we linearly interpolate $\widetilde{\chi_{H}}^n(\mu)$ on the fine mesh in order to compute the $L^2$-inner product with the basis functions. The approximation used for the adjoint in the two-grid method is
      \begin{equation}
        \label{NIRBapproximation00ad}
     \textrm{For }n=0,\dots, N_T, \quad   \chi_{Hh}^{N_1,n}(\mu):= \overset{N_1}{\underset{i=1}{\sum}}(\widetilde{\chi_{H}}^n(\mu),\xi^{h}_{i})\ \xi^{h}_{i},
      \end{equation}
      and with the rectification post-treatment step, it becomes
         \begin{equation}
        \label{NIRBapproximation0rect0ad}
       \mathbf{R}^{n}_{\chi}[\chi_{Hh}^{N_1}](\mu):= \overset{N_1}{\underset{i,j=1}{\sum}}R_{\chi,ij}^{n}\ (\widetilde{\chi_{H}}^n(\mu),\xi^{h}_{j})\ \xi^{h}_{i},
         \end{equation}
         where $\mathbf{R}^{n}_{\chi}$ is the rectification matrix at time $t^n$, given by \eqref{rectiff5}.
         \item Then, we use the steps 5 and 6 of {}{S}ection \ref{NIRBParabolic} in order to obtain a NIRB approximation for $u(\mu)$ from the coarse solution $u_H^m$ given by step 4 of this online part. 
         \item Finally, the sensitivities NIRB approximations of $\mathcal{F}$ are given by 
              \begin{equation}
        \label{NIRBapproximation00sensF}
    \textrm{for }p=1,\dots,P,\  [\frac{\pt \mathcal{F}}{\pt \mu_p}]_{Hh}^{N_1}(\mu):= \overset{t^n}{\underset{j=1}{\sum}} \Delta t_F\ \Big(\chi_{Hh}^{N_1,j}, \nabla \cdot ( \frac{\pt A}{\pt \mu_p} (\mu)\nabla u_{Hh}^{N,j} )\big), \textrm{ from \eqref{Dmup1}},
      \end{equation}
      and with the rectification post-treatment step, it becomes
         \begin{equation}
        \label{NIRBapproximation0rect0sensF}
  \textrm{for }p=1,\dots,P,  \ \mathbf{R}_{\chi}[  [\frac{\pt \mathcal{F}}{\pt \mu_p}]_{Hh}^{N_1}](\mu):= \overset{t^n}{\underset{j=1}{\sum}}  \Delta t_F\ \Big(\mathbf{R}^j_{\chi}[\chi_{Hh}^{N_1,j}](\mu), \nabla \cdot (\frac{\pt A}{\pt \mu_p}(\mu)\nabla \mathbf{R}^j_u[u_{Hh}^{N,j}](\mu)) \big).
         \end{equation}
\end{enumerate}
 \end{itemize}
The next section gives our main result on the NIRB two-grid method error estimate in the context of sensitivity analysis.

\section{Numerical results.}
\label{results}
 In this section, we {}{apply} the NIRB algorithms on several numerical tests.
  We have implemented the problem model and the Brusselator system (as our two NIRB applications) using FreeFem++ (version 4.9) \cite{bamg2} to compute the fine and coarse snapshots, and the solutions have been stored in VTK format.
  
Then we have applied the NIRB algorithms with python (with the library scikit-learn \cite{kramer2016scikit} for the GP version), in order to highlight the non-intrusive side of the two-grid method (as in \cite{paraboliccontext}). After saving the NIRB approximations with Paraview module in Python, the errors have been computed with FreeFem++. {}{The codes are availbale on GitHub\footnote{https://github.com/grosjean1/SensitivityAnalysisWithNIRBTwoGridMethod}}.
\subsection{On the heat equation.}
In order to validate our numerical results, we have chosen a right-hand side function
    such that for $\mu = 1$, we can calculate an analytical solution for $u$, which is given by
\begin{equation}
    u(t, \xx; 1) = 10 (t+1) x^2(1 - x)^2y^2(1 - y)^2, \forall t \in [0,1],
    \end{equation}
where $ \xx = (x, y) \in [0,1]^2$.
We have solved \eqref{heatEq2} and \eqref{varparasensheat} on the parameter set $\mathcal{G} = [0.5, 9.5]$. The initial solution $u_0(\mu)$ solves the elliptic equation (with homogeneous Dirichlet boundary conditions)
\begin{equation*}
    - \mu \Delta u_0=f_0, \textrm{ with }f_0(\xx)=- 20\ [(6x^2 - 6x + 1)(y^2(y - 1)^2) + (6y^2 - 6y + 1)(x^2(x - 1)^2)].
    \end{equation*}
We have retrieved several snapshots on $t = [0, 1]$ (note that the coarse time grid must belong to the interval of the fine one), and tried our algorithms on several size of meshes, always with $\Delta t_F\simeq h$ and $\Delta t_G \simeq H$ (both schemes are stable), and such that $h=H^2$. 
\begin{itemize}
\item \figref{fig:fig1} illustrates the convergence rates in $\ell^{\infty}(0,\dots, N_T;H_0^1(\Omega)$ of the FEM and NIRB {}{sensitivity} approximations. We have taken 18 parameters in $\mathcal{G}$ for the RB construction such that $\mu_i = 0.5i, \ i = 1, \dots , 19, i \neq 2 $ and a reference solution to problem \eqref{eq:directheatsensitivity}, with $\mu = 1$ and its mesh and time step such that $h_{ref}\simeq \Delta t_{F,ref}=0.0025$. We have compared the FEM errors with coarse and fine grids to the ones obtained with the NIRB algorithms (classical NIRB, NIRB with rectification and NIRB with GPR, {}{all described in Section} \ref{NIRBalgosensdirect}).
\begin{figure}
    \centering
    \includegraphics[scale=0.3]{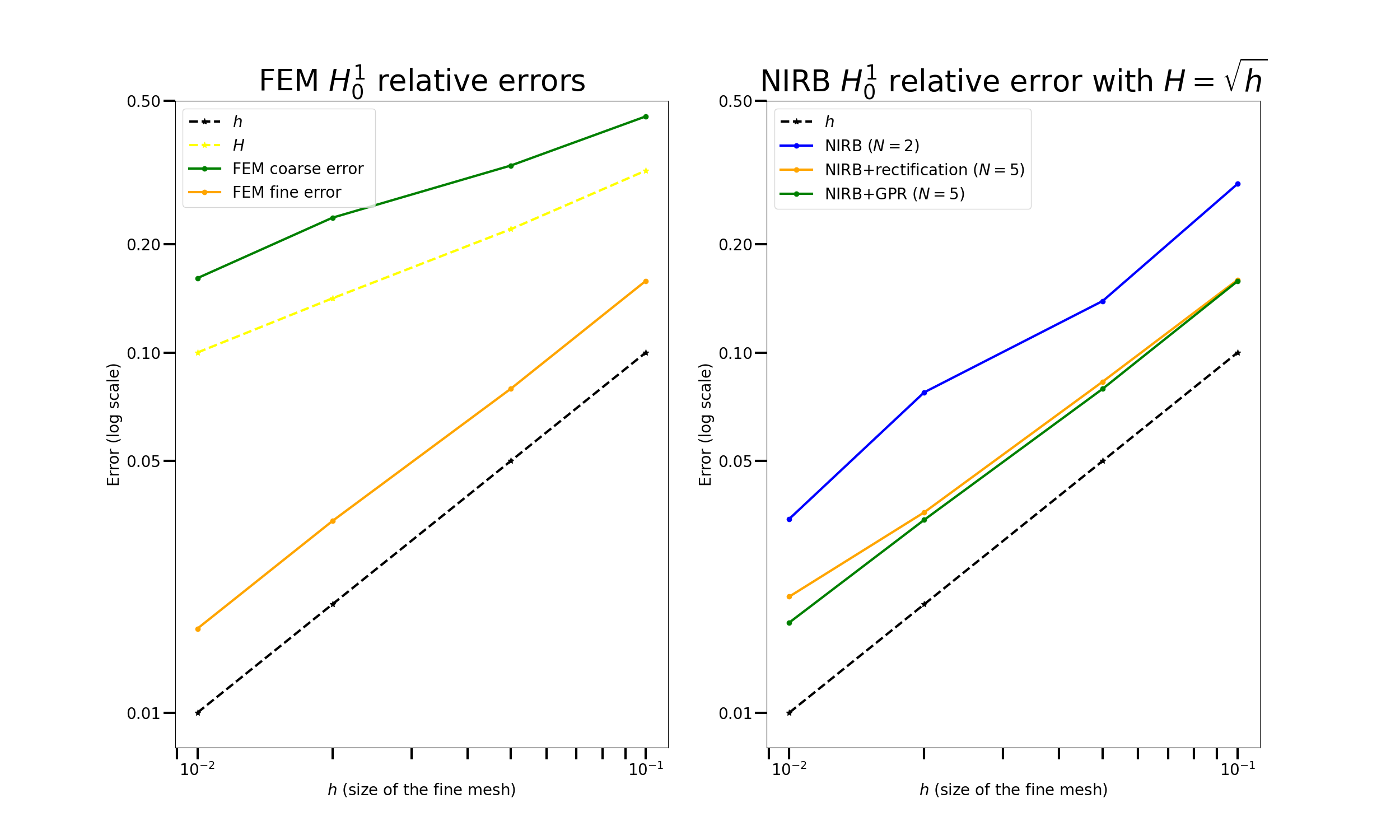}
    \caption{{}{Direct sensitivities -} heat equation: $H^1_0$ relative FEM errors (left) compared to $H^1_0$ relative NIRB errors (right) with several $\Delta t_F \simeq h \simeq H^2 \simeq \Delta t_G $ and $\mu=1$ }
    \label{fig:fig1}
\end{figure}

\item  Then, we have taken 19 parameters in $\mathcal{G}$ for the RB construction such that $\mu_i=0.5i, \ i=1,\dots, 19$ and have applied the ``leave-one-out'' strategy in order to evaluate the NIRB algorithm with respect to the parameters (whereas \figref{fig:fig1} illustrates results with $\mu=1$). Table \ref{TableParam1} presents the maximum relative $H^1_0$-errors between the {}{true} solution $\Psi_{p,h}(\mu)$ and the NIRB approximations over the parameters (inside the parameters training set). The plain NIRB error is defined by 
 \begin{equation}
 \label{plainNIRBeq}
\underset{\mu \in \mathcal{G}_{train}}{\max} \frac{\norm{\Psi_{p,h_{ref}}(\mu) - \Psi_{p,Hh}^{N_p}(\mu)}_{\ell^{\infty}(0,\dots,N_T;H^1_0(\Omega))}}{\norm{\Psi_{p,h_{ref}}(\mu)}_{\ell^{\infty}(0,\dots,N_T;H^1_0(\Omega))}},
 \end{equation}
 and for the NIRB rectified error $\Psi_{p,Hh}^{N_p}(\mu)$ is replaced by $\mathbf{R}_{\Psi}[\Psi_{p,Hh}^{N_p}](\mu)$ in \eqref{plainNIRBeq}, and by $\Psi_{p,Hh}^{GP}(\mu)$ for the GP version (respectively defined by \eqref{NIRBapproximation0rect0} and \eqref{NIRBapproximation0rect00}. 
 We compare these errors to the ones obtained with the true projection on $X_{p,h}^{N_p}$, denoted by $ \Psi_{p,hh}^{N_p,n}(\mu)$ and given by 
      \begin{equation}
      \label{errorHFprojection}
     \quad   \Psi_{p,hh}^{N_p,n}(\mu):= \overset{N_p}{\underset{i=1}{\sum}}(\Psi_{p,h}^n(\mu),\zeta^{h}_{p,i})\ \zeta^{h}_{p,i}, \quad  \textrm{for }n=0,\dots, N_T.
      \end{equation}

 \begin{table}[tbhp]
     {\footnotesize
       \caption{ {}{Direct formulation: Maximum relative $\ell^{\infty}(0,\dots,N_T;H^1_0(\Omega)$ error over the parameters} (compared to the true projection and to the FEM coarse projection) with $N=5$ with several $h\simeq \Delta t \simeq H^2\simeq \Delta t_G^2$}\label{TableParam1}
   \begin{center}
 {   \begin{tabular}{ |c|c|c|c| c| c|}
      \hline
$\Delta t_F$-$\Delta t_G$&  Plain NIRB & NIRB rectified & NIRB-GP  & $ \underset{\mu \in \mathcal{G}_{train}}{\max} \frac{\norm{\Psi_{p,h_{ref}}(\mu) - \Psi_{p,hh}^{N_p}(\mu)}_{\ell^{\infty}(0,\dots,N_T;H^1_0(\Omega))}}{\norm{\Psi_{p,h_{ref}}(\mu)}_{\ell^{\infty}(0,\dots,N_T;H^1_0(\Omega))}}$ &   $\underset{\mu \in \mathcal{G}_{train}}{\max} \frac{\norm{\Psi_{p,h_{ref}}(\mu) - \Psi_{p,H}(\mu)}_{\ell^{\infty}(0,\dots,M_T;H^1_0(\Omega))}}{\norm{\Psi_{p,h_{ref}}(\mu)}_{\ell^{\infty}(0,\dots,M_T;H^1_0(\Omega))}}$ \\
    \hline
  0.01-0.1&$3.79\times 10^{-2}$&$1.71\times 10^{-2}$  & $ 1.71 \times 10^{-2}$ & $1.69\times 10^{-2}$&$1.56\times 10^{-1}$ \\
  \hline
  0.02-0.1414&$8.24\times 10^{-2}$ &$3.49\times 10^{-2}$ & $3.39 \times 10^{-2}$ &$3.38\times 10^{-2}$  &   $2.32\times 10^{-1}$\\
  \hline
 0.05-0.22 &$1.38\times 10^{-1}$ & $7.89\times 10^{-2}$& $7.89\times 10^{-2}$ & $7.89\times 10^{-2}$&   $3.24\times 10^{-1}$\\
 \hline
    0.1-0.32& $2.99\times 10^{-1}$&$1.58 \times 10^{-1}$& $1.57 \times 10^{-1}$ &  $1.57\times 10^{-1}$ &  $4.38 \times 10^{-1}$\\
    \hline
    \end{tabular}}
   \end{center}
   }
 \end{table}
 
 \item Finally, to demonstrate the great capability of the NIRB two-grid method to approach a solution of the adjoint sensitivity equations, we have tested the backward approach on several parameters (included in the training set $\mathcal{G}$) with a leave-one-out strategy and we present the results in Table \ref{TableParamAdj1} and \ref{TableParamAdj1bis}, with and without noisy measurements for $\overline{u}$ (where for the noisy version, we have added a Gaussian noise centered in $0$ and of standard deviation $\sigma=0.1$). We see that, with the NIRB rectified algorithm, we retrieve the accuracy of the fine adjoint solutions in both cases.

  \begin{table}[tbhp]
     {\footnotesize
       \caption{ {}{{}{Adjoint formulation without noise:} Maximum absolute $\ell^{\infty}(0,\dots,N_T;H^1_0(\Omega))$ error over the parameters} (compared to the true projection and to the FEM fine and coarse projections) with $N=15$ with several $h\simeq \Delta t \simeq H^2\simeq \Delta t_G^2$} (with $h_{ref} \simeq \Delta t_{F,ref}=0.0025$ for the measurements $\overline{u}$ and with a reference solution for the adjoint $\chi_{ref}$ where $h_{ref} \simeq \Delta t_{F,ref}=0.005$) \label{TableParamAdj1}
   \begin{center}
   { \begin{tabular}{ |c|c|c|c|c|}
      \hline
$\Delta t_F$-$\Delta t_G$& NIRB rectified error & maximum absolute fine error &  maximum absolute true projection error   & maximum absolute coarse error \\
    \hline
  0.01-0.1& $4.9 \times 10 ^{-6}$ &  $7.0 \times 10 ^{-6}$&  $7.0 \times 10 ^{-6}$ &  $3.3\times 10^{-4}$\\
  \hline
  0.02-0.1414& $9.1 \times 10^{-6}$ & $8.1 \times 10^{-6}$& $8.1 \times 10 ^{-6}$ &$7.0\times 10^{-4}$\\
  \hline
 0.05-0.22 &$1.2 \times 10^{-4}$ & $1.0 \times 10^{-4}$&  $1.0 \times 10 ^{-4}$&$1.3\times 10^{-3}$\\
    \hline
    0.1-0.32&$ 2.2 \times 10^{-4}$&  $1.7 \times 10^{-4}$& $1.7 \times 10 ^{-4}$& $2.3\times 10^{-3}$ \\
    \hline
    \end{tabular}}
   \end{center}
   }
 \end{table}

   \begin{table}[tbhp]
     {\footnotesize
       \caption{ {}{{}{Adjoint formulation with noise:} Maximum absolute $\ell^{\infty}(0,\dots,N_T;H^1_0(\Omega))$ error over the parameters} (compared to the true projection and to the FEM coarse projection) with $N=15$ with several $h\simeq \Delta t \simeq H^2\simeq \Delta t_G^2$ and noisy measurements} (with $h_{ref} \simeq \Delta t_{F,ref}=0.0025$ for the measurements $\overline{u}$ and with a reference solution for the adjoint without noise $\chi_{ref}$ where $h_{ref} \simeq \Delta t_{F,ref}=0.005$) 
       \label{TableParamAdj1bis}
   \begin{center}
   { \begin{tabular}{ |c|c|c|c|c|}
      \hline
$\Delta t_F$-$\Delta t_G$& NIRB rectified error & maximum absolute fine error &  maximum absolute true projection error   & maximum absolute coarse error \\
    \hline
  0.01-0.1&$6.7 \times 10 ^{-4}$  & $6.6 \times 10 ^{-4}$ & $6.7 \times 10 ^{-4}$   & $4.0\times 10^{-3}$\\
  \hline
  0.02-0.1414&$1.4\times 10^{-3}$& $1.3 \times 10^{-3}$ & $1.3 \times 10^{-3}$ &$6.0\times 10^{-3}$\\
  \hline
 0.05-0.22 &$6.3 \times 10^{-3}$&$6.3 \times 10^{-3}$ &  $6.3 \times 10^{-3}$&$8.1\times 10^{-3}$\\
    \hline
    0.1-0.32&$ 3.2 \times 10^{-3}$&  $3.0 \times 10^{-3}$& $3.0 \times 10^{-3}$& $1.0\times 10^{-2}$ \\
    \hline
    \end{tabular}}
   \end{center}
   }
 \end{table}

\end{itemize}


\subsection{On the Brusselator system.}
{}{In this section, we shall consider the nonlinear rate equations of the trimolecular model or Brusselator. It is used for the study of cooperative processes in chemical kinetics \cite{prigogine1985self,MITTAL20115404}.}
It is a more complex test from a simulation point of view than the heat equation due to its nonlinearity. 
The chemical concentrations $u_1$ and $u_2$ are controlled by parameters $\boldsymbol{\mu}=(a,b,\alpha)$ throughout the reaction process, making it an interesting application of a NIRB method in the context of sensitivity analysis.\\
Let us consider as the spatial domain $\Omega=[0,1]^2$. The nonlinear system of this two-dimensional reaction-diffusion problem with $\mathbf{u}=(u_1,u_2)$ {}{is}:
\begin{equation*}
\begin{cases}
   & \pt_t u_1  = F_1(\mathbf{u}):=a+ u_1^2 u_2 -(b+1) u_1 + \alpha \Delta u_1,\textrm{ in }\Omega\times ]0,T] \\
    & \pt_t u_2  = F_2(\mathbf{u}):=bu_1 -u_1^2 \ u_2 +\alpha  \Delta u_2, \textrm{ in } \Omega \times ]0,T], \\
      &  u_1(\xx,0)=u^0(\xx)=2+0.25y,  \textrm{ in } \Omega \\
            & u_2(\xx,0)=v_0(\xx)=1 + 0.8x,  \textrm{ in } \Omega, \\
            & \pt_n u_1 =0 ,\ \pt \Omega,\\
  & \pt_n u_2 = 0,\  \pt \Omega.
\end{cases}
\end{equation*}
Our parameter, denoted $\boldsymbol{\mu}=(a,b,\alpha)$, belongs to {}{$\mathcal{G}:=[2,5]\times [1,5] \times [0.0001,0.05]$}. These parameters are standard {}{\cite{MITTAL20115404}}. We have taken a {}{final} time $T=4$. We note that for $b\leq 1 + a^2$ the solutions are stable, and for $\alpha$ small enough they converge to $(u_l,v_l)=(a,\frac{b}{a})$.\\
We use an {}{implicit} Euler scheme for fine solutions and fine sensitivities with the Newton algorithm to deal with {}{the} nonlinearity and {}{the} Crank-Nicolson scheme for the coarse mesh.\\
 \begin{itemize}
\item  {}{We have tested our NIRB algorithms on several parameters in $\mathcal{G}$ with a leave-one-out strategy, and 34 training parameters for the RB} (see \cite{maday2013locally} for a Greedy algorithm with adaptive choice of optimal training set, adapted to a target accuracy).
We have employed a refined mesh to represent the solution of reference (with $h_{ref}=\Delta t_{ref}=0.005$). In Tables \ref{tablebruss} and \ref{tablebruss2}, we have compared the $\ell^{\infty}(0,\dots,N_T;\ H^1(\Omega))$ error of the fine FEM solutions to the corresponding NIRB errors with the rectification post-treatment and with the GP process, with $40$ modes and $h=0.02 =~\Delta t_F $ and $ H = \Delta t_G=0.1$. {}{For the GP kernel, we have employed a \textit{Dot Product} kernel \cite{lu2005positive} which takes the form \begin{equation*}
    \kappa(\xx,\xx')=1+\xx\cdot \xx'. 
\end{equation*}
}
\small{
\begin{table}[h]
    \centering
{    \begin{tabular}{|c|c|c|c|c|c|}
  \hline
{}{   Parameters $a$-$b$-$\alpha$}&Fine error&Coarse error& True projection $u_{hh}^N$ & NIRB + rectification & NIRB+GPR  \\
\hline
\hline
{}{
3-2-0.01}& $5.6 \times 10^{-2}$ &   $2.8\times 10^{-1}$& $5.6 \times 10^{-2}$  &$5.6 \times 10^{-2}$  & $ 6.0 \times 10^{-2}$\\ 
\hline
{}{3-3-0.01}&$5.9 \times 10^{-2}$&  $2.8\times 10^{-1}$&$5.9 \times 10^{-2}$ &$5.9 \times 10^{-2}$ & $ 7.0\times 10^{-2}$ \\
\hline
{}{3-4-0.01 }& $9.2 \times 10^{-2}$& $3.8\times 10^{-1}$ & $9.2 \times 10^{-2}$& $9.2 \times 10^{-2}$& $1.5\times 10^{-1}$\\
\hline
{}{4-2-0.0005} &$7.3\times 10^{-2}$ & $3.3\times 10^{-1}$ & $7.3\times 10^{-2}$ &$7.3\times 10^{-2}$  &$6.0\times 10^{-2}$ \\
\hline
{}{4-3-0.0005}& $6.7 \times 10^{-2}$& $3.1\times 10^{-1}$ &$6.7 \times 10^{-2}$ &$6.7 \times 10^{-2}$ &$7.0\times 10^{-2}$ \\
\hline
{}{4-4-0.0005}&$8.4\times 10^{-2}$ & $3.7\times 10^{-1}$&$ 8.4\times 10^{-2}$&$8.4\times 10^{-2}$&$9.2 \times 10^{-2}$ \\
\hline
    \end{tabular}}
    \caption{ {}{relative $\ell^{\infty}(0,\dots,N_T;\ H^1(\Omega))$ errors (and  $\ell^{\infty}(0,\dots,M_T;\ H^1(\Omega))$ for the coarse ones) with leave-one-out strategy $N=40$ for the parameter $a$}} 
    \label{tablebruss}
\end{table}
}

\small{
\begin{table}[h]
    \centering
   { \begin{tabular}{|c|c|c|c|c|c|}
    \hline
{}{Parameters $a$-$b$-$\alpha$}&Fine error&Coarse error& True projection $u_{hh}^N$ & NIRB + rectification & NIRB+GPR  \\
\hline
\hline
{}{3-2-0.01}& $ 6.4\times 10^{-2}$ &   $3.1\times 10^{-1}$& $6.4\times 10^{-2}$  & $6.4\times 10^{-2}$ &$ 6.8\times 10^{-2}$ \\ 
\hline
{}{3-3-0.01}&$ 6.2\times 10^{-2}$&  $3.0\times 10^{-1}$&$ 6.2\times 10^{-2}$ & $ 6.2\times 10^{-2}$& $ 6.9\times 10^{-2}$ \\
\hline
{}{3-4-0.01} & $ 8.6\times 10^{-2}$& $3.9\times 10^{-1}$ & $8.6\times 10^{-2}$& $8.6\times 10^{-2}$& $ 1.4\times 10^{-1}$ \\
\hline
{}{4-2-0.0005 }&$ 7.3\times 10^{-2}$ & $3.3\times 10^{-1}$ &$ 7.3\times 10^{-2}$ &$ 7.3\times 10^{-2}$& $ 6.8\times 10^{-2}$\\
\hline
{}{4-3-0.0005}& $ 7.1\times 10^{-2}$& $3.3\times 10^{-1}$ &$7.1\times 10^{-2}$  &$7.1\times 10^{-2}$ &$7.7\times 10^{-2}$ \\
\hline
{}{4-4-0.0005}&$ 8.5\times 10^{-2}$ & $3.8\times 10^{-1}$& $8.5 \times 10^{-2}$&$8.5 \times 10 ^{-2}$ &$ 9.2\times 10^{-2}$ \\
\hline
    \end{tabular}}
    \caption{ {}{relative $\ell^{\infty}(0,\dots,N_T;\ H^1(\Omega))$ errors (and  $\ell^{\infty}(0,\dots,M_T;\ H^1(\Omega))$ for the coarse ones) with leave-one-out strategy $N=40$ for the parameter $b$} }
    \label{tablebruss2}
\end{table}
}

\item {In order to observe the effect of the number of modes $N$ on the NIRB-GP approximation, we plot its errors for several $N$ and the ones obtained from the FEM fine and coarse approximations in Figure \ref{Bruss1} (sensitivities of $a$) and compare them to the true projection for the worst-case scenario (with $a=3$, $b=4$ and $\alpha=0.01$). In Figure \ref{plots}, we present the sensitivities with the fine HF solver and with the NIRB-GP approach, after 100 time steps.

   \begin{figure}
      \centering
\includegraphics[scale=0.3]{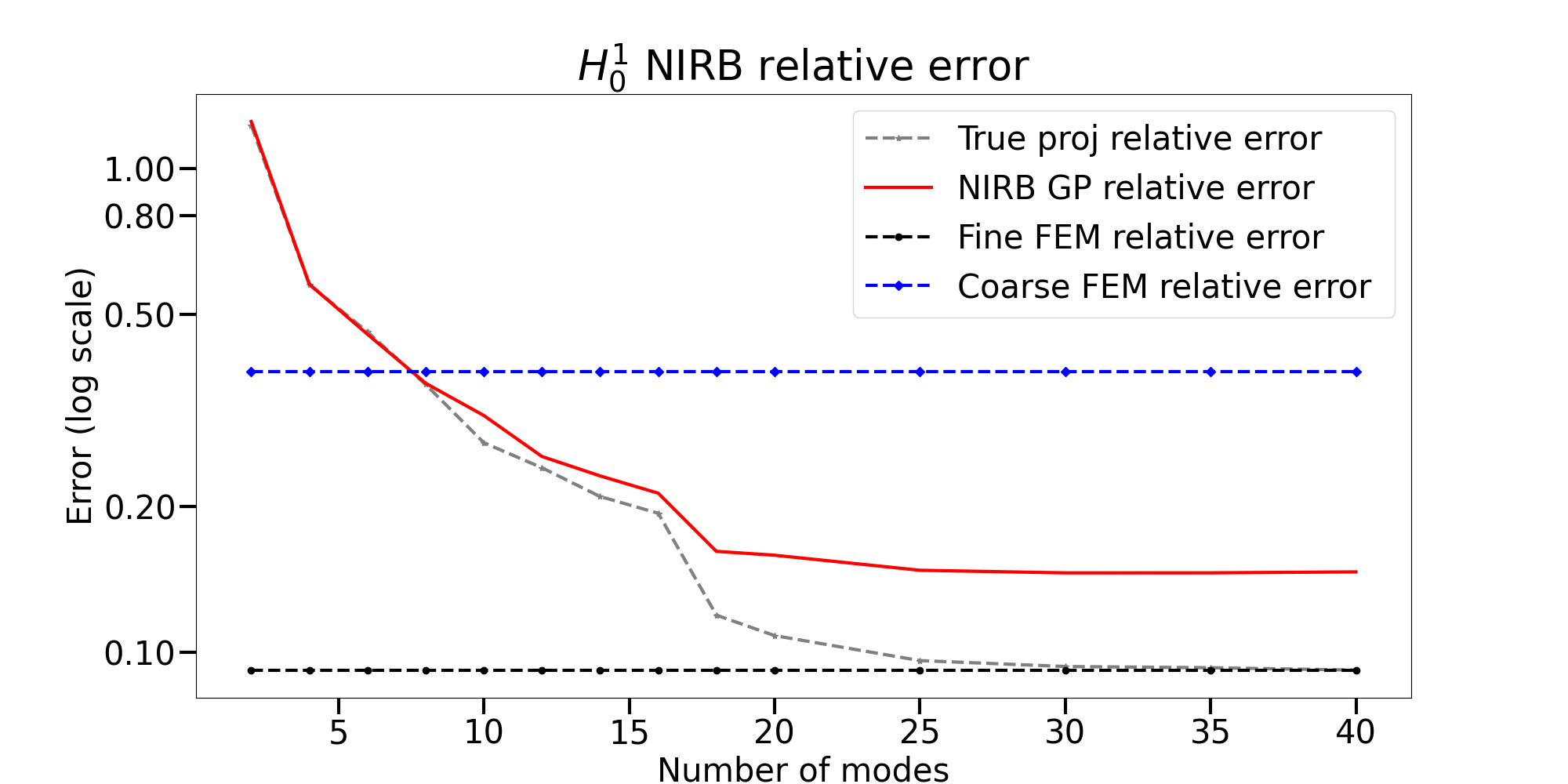}
\caption{{}{Sensitivities $\ell^{\infty}(0,\dots,N_T;\ H^1(\Omega))$  relative errors with respect to time $t=[0,4]$, with the worst-case scenario where the parameter is $(a,b,\alpha)=(3,4,0.01)$: NIRB-GP version compared to the true RB-projection }}
\label{Bruss1}
\end{figure}}

  \begin{figure}[!htbp]
     \centering
     \begin{subfigure}[b]{0.45\textwidth}
         \centering
         \includegraphics[scale=0.2]{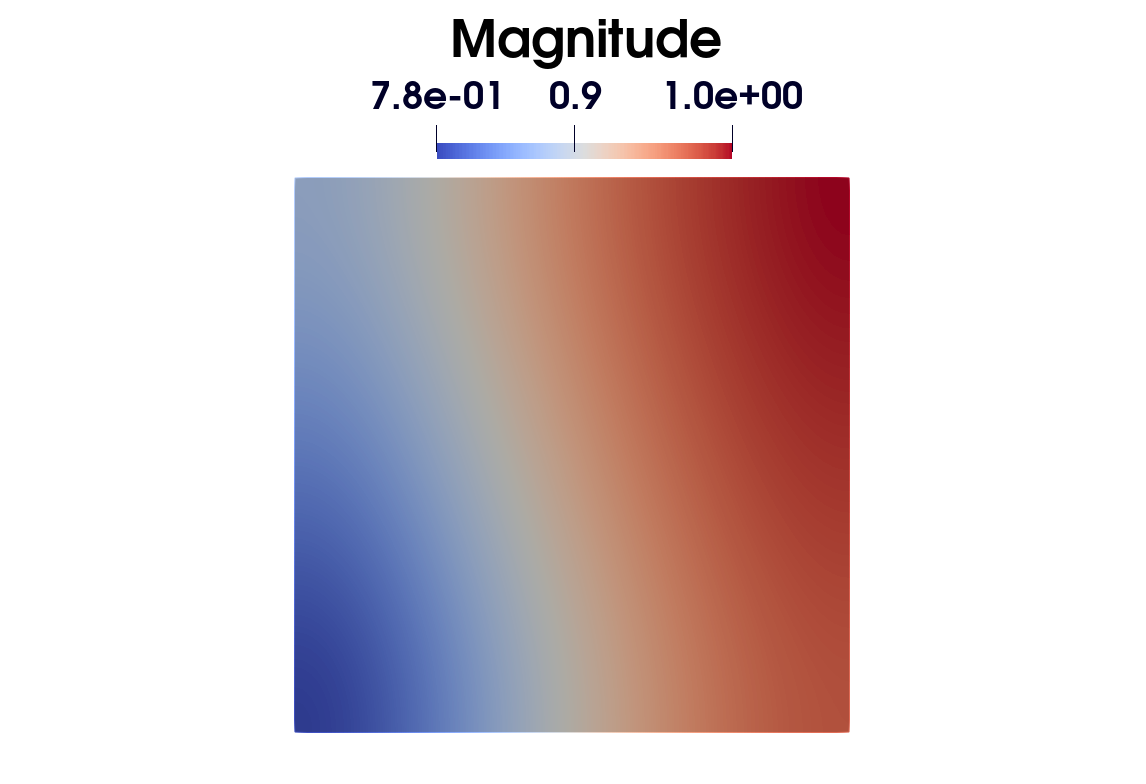}
         \caption{With fine HF solver}
     \end{subfigure}
     \quad \quad  
     \begin{subfigure}[b]{0.45\textwidth}
         \centering
         \includegraphics[scale=0.2]{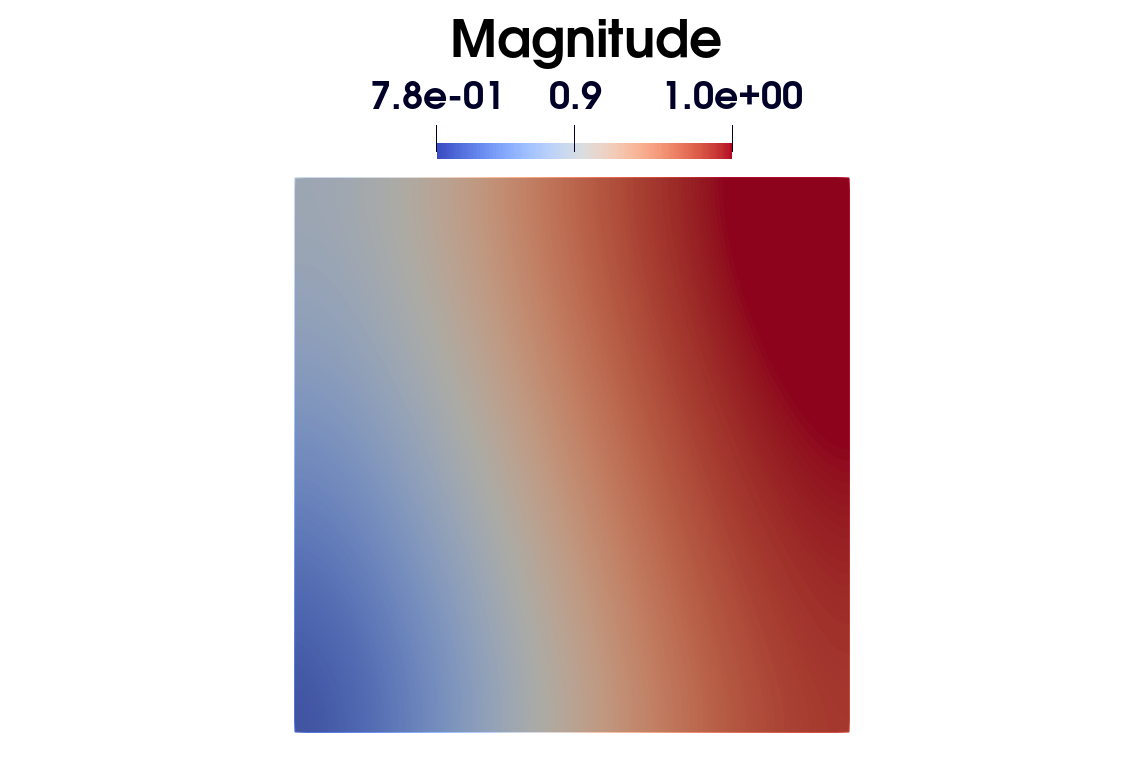}
         \caption{With NIRB-GP approximation}
     \end{subfigure}
     \quad \quad \;
 \caption{{}{Magnitude of the sensitivities, with the worst-case scenario where the parameter is $(a,b,\alpha)=(3,4,0.01)$: Fine HF solution vs NIRB-GP version, after 100 time steps }}
 \label{plots}
   \end{figure}

\item Then, we have tested the NIRB algorithm with the adjoint formulation \ref{NIRBalgosensadjoint}, where we have used as the objective function $\mathcal{F}(\mu)= \frac{1}{2} \overset{N_T}{\underset{n=0}{\sum}} \norm{\mathbf{u}_{h}^n(\mu) - \overline{\mathbf{u}}^n}^2_{L^2(\Omega)}$ as in \eqref{Fobj} (or \eqref{Fobjcontinuous} in its continuous version). We consider noisy {}{measurements} $\overline{\mathbf{u}}^n$. We display in Table \ref{TableParam1} the evaluations of the sensitivity $\frac{\pt \mathcal{F}}{\pt a}(\mu)$ with the fine sensitivities (respectively with $b$ and $\alpha$) and compare them with the NIRB ones given by \eqref{NIRBapproximation0rect0sensF}.
Now, the Lagrangian for the  Brusselator system with $\boldsymbol{\chi}=(\chi_1,\chi_2)$ as our {}{lagrange multiplier} takes the form: \\
\begin{equation}
    \mathcal{L}(\mathbf{u},\boldsymbol{\chi}; \mu)= \mathcal{F}(\mu)+ \int_0^T (\chi_1,F_1(\mathbf{u})- {u_1}_t ) + (\chi_2,F_2(\mathbf{u}) - {u_2}_t) \ ds.
\end{equation}
We obtain the following associated adjoint system (in its continuous version)
 \begin{numcases}{}
 \label{strongformchi}
&$  \chi_{1t}= - \frac{\pt err}{\pt u_1} - 2 u_1 u_2 \chi_1 + (b+1) \chi_1 - \alpha \Delta \chi_1 - b \chi_2 + 2 u_1 u_2 \chi_2, \ \textrm{ in }\Omega\times [0,T[$, \notag \\
&$  \chi_{2t}= - \frac{\pt err}{\pt u_2} - \chi_1 u_1^2 + \chi_2 u_1^2 - \alpha \Delta \chi_2, \ \textrm{ in }\Omega\times [0,T[$, \notag \\
      & $\chi_1(\cdot,T) = 0,\ \textrm{ in }\Omega,$ \notag \\
      & $\chi_2(\cdot,T) = 0,\ \textrm{ in }\Omega,$ \notag \\
 & $\pt_n \chi_1 (\cdot,t) = 0, \ \textrm{ on }\pt \Omega \times [0,T[$, \notag \\
  & $\pt_n \chi_2 (\cdot,t) = 0, \ \textrm{ on }\pt \Omega \times [0,T[$.
\end{numcases}
  \end{itemize}
   

 \begin{table}[tbhp]
     {\footnotesize
       \caption{ Objective sensitivity values with $N=40$ with $h=0.02 \simeq H^2$}\label{TableParam1}
   \begin{center}
    \begin{tabular}{ |c|c|c| }
      \hline
  Parameter & High-Fidelity values & rectified NIRB adjoint version \\
    \hline
  $a $ & $-2.74\times 10^{-6}$ & $5.70 \times 10^{-5}$\\
    \hline
    $b$ & $1.17\times 10^{-5}$ & $1.66 \times 10^{-4}$\\
    \hline 
    $\alpha$ & $ -6.73\times 10^{-4}$ & $-3.8 \times 10^{-2}$ \\
    \hline
    \end{tabular}
   \end{center}
   }
 \end{table}

\subsubsection{ Time execution (min,sec)}
\label{time3}
Finally, the computational costs are reduced well during the online part of these NIRB algorithms as it is highlighted with this example ($h=0.02$ and $H=0.1$). We present the FEM  forward simulation runtimes for one instance of parameter in {}{T}able \ref{runtimes13} and the NIRB runtimes with the adjoint rectified and with the GP versions in Table \ref{runtimes23} (offline and online times).
   \begin{table}[tbhp]
     {\footnotesize
       \caption{ FEM runtimes (h:min:sec)}\label{runtimes13}
   \begin{center}
    \begin{tabular}{ |c| c| }
      \hline
   FEM high fidelity solver & FEM coarse solution \\
    \hline
   00:07:05 & 00:00:58\\
    \hline
    \end{tabular}
   \end{center}
   }
   \end{table}
 \begin{table}[tbhp]
     {\footnotesize
       \caption{NIRB runtimes ($N=40$, h:min:sec)}\label{runtimes23}
   \begin{center}
{    \begin{tabular}{|c|c|c|}
      \hline
       & Offline & Online\\
      \hline
    Rectified NIRB &  01:25:00 & 00:04:02 \\
      \hline
      NIRB-GP &   01:32:00 & 00:03:57\\
      \hline
    \end{tabular}}
   \end{center}
   }
   \end{table}
  
   { \begin{remark}
    On the parameter $b$, we present in Table \ref{lambdaN} the values of $\lambda(N)^{1/2}$ for different number of modes, and we refer to \cite{paraboliccontext} for more details on why we can assume that it remains bounded.
        \begin{table}[tbhp]
     {\footnotesize
       \caption{Values of $\lambda(N)^{1/2}$ for $N=1$ to $N=40$}\label{lambdaN}
   \begin{center}
   { \begin{tabular}{|c|c|c|c|c|c|c|c|c|c|c|}
      \hline
    $ 2.4 \times 10^{-2}$&$1.1 \times 10^1$&$ 6.1 \times 10^1$&$8.7\times 10^1 $&$1.2 \times 10^2$& \dots &  $ 3.4\times 10^2$ & $ 3.4\times 10^2$ & $ 3.4\times 10^2$ & $ 3.4\times 10^2$ & $ 3.4\times 10^2$\\
      \hline
    \end{tabular}}
   \end{center}
   }
   \end{table}
 \end{remark}}
 
   \appendix
   \section{Appendix.}
     \subsection{Proof of Theorem \ref{corollaryPsii}}
   \label{proofEuler}
     \begin{proof}
  As in \cite{thomee2}, we first decompose the error with two components $\theta^n$ and $\rho^n$  such that, on the discretized time grid $(t^n)_{n=0,\dots,N_T}$,
\begin{align}
\forall n=0,\dots,N_T, \ e^n:= \Psi_h^n- \Psi(t^n)&=(\Psi_h^n- P^1_h \Psi(t^n))+(P^1_h \Psi(t^n)- \Psi(t^n)), \nonumber \\
&:=\theta^n + \rho^n.
\label{errorParaboldisc}
\end{align}

\begin{itemize}
\item The estimate on $\rho^n$ (see Lemma 1.1 \cite{thomee2}) leads to
\begin{equation}
  \norm{\rho^n}_{L^2(\Omega)}\lesssim h^2 \norm{\Psi(t^n)}_{H^2(\Omega)} \leq h^2 \big[\norm{\Psi^0}_{H^2(\Omega)} + \int_0^{t^n} \norm{ \Psi_t}_{H^2(\Omega)} \ \textrm{ds} \big], \  \forall n \in \{0,\dots,N_T\}.
\label{rhoclassiquedisc}
\end{equation}
\item In order to derive the estimate of $\theta^n$, from the weak formulations \eqref{eq:2chainrule} and \eqref{eq:directheatsensitivity}, by definition of $P^1_h$ \eqref{P1opParabol} and choosing $v=\theta_h$, we remark that
\begin{align}
  \label{secondexpression}
    (\overline{\pt} \theta^n,\theta^n)+\mu \norm{\nabla \theta^n}_{L^2(\Omega)}^2  &=\frac{1}{\mu}(\overline{\pt} u_h^n-  u_t(t^n), \theta^n) - (w^n, \theta^n), \notag \\
&=\underbrace{\frac{1}{\mu}(\overline{\pt} u_h^n- u_t(t^n), \theta^n)}_{e^u} - (w_1^n+w_2^n, \theta^n), 
\end{align}
where $w_1^n$ and $w_2^n$ are defined by
\begin{equation}
  \label{w1etw2}
  w_1^n:=(P_h^1-I)\overline{\pt} \Psi(t^n), \textrm{\quad   } w_2^n:=\overline{\pt} \Psi(t^n) - \Psi_t(t^n), \quad \textrm{and} \quad w^n:=w_1^n+w_2^n.
  \end{equation}
The weak formulation \eqref{eq:2chainrule} corresponds to the one of the heat equation with $\Psi$ as the unknown and a source term dependant of the solution $u(\mu)$. An estimate similar for $\theta$ is proved in \cite{thomee2} for a source term $f\in L^2(\Omega \times [0,T])$ and $\mu=1$ such that the discretized version of $f$ is the evaluation of the continuous function $f$ on the discretization points. Thus, only the contribution of $e^u$ in \eqref{secondexpression} is new and remains to be estimated.  
  
As a matter of fact, by definition of $\overline{\pt}$ and by Cauchy-Schwarz inequality (and since the second term of the left-hand side of \eqref{secondexpression} is always positive), 
\begin{equation*}
  \label{thirdexpression}
  \norm{\theta^n}_{L^2(\Omega)}^2 \leq \big(\norm{\theta^{n-1}}_{L^2(\Omega)}+\Delta t_F \big[ \frac{1}{\mu} \norm{\overline{\pt} u_h^n -  u_t(t^n)}_{L^2(\Omega)} + \norm{w^n}_{L^2(\Omega)} \big] \big) \norm{\theta^n}_{L^2(\Omega)},
\end{equation*}
and by repeated application, and since  $\norm{\theta^0}_{L^2(\Omega)}=0$, it entails
\begin{equation}
  \label{thirdexpression}
  \norm{\theta^n}_{L^2(\Omega)}\leq  \underbrace{\Delta t_F \underset{j=1}{\overset{n}{\sum}} \frac{1}{\mu} \norm{\overline{\pt} u_h^j - u_t(t^j)}_{L^2(\Omega)}}_{T_{1,n}} + \underbrace{\Delta t_F \underset{j=1}{\overset{n}{\sum}} \norm{w^j}_{L^2(\Omega)}}_{T_{2,n}},
\end{equation}
\begin{itemize}
\item 
To bound $T_{1,n}$, as for the unknown $\Psi$, we can define 
\begin{equation}
\label{errorParabolU}
    \theta_u^n:=(u_h(t)- P^1_h u(t)),
\end{equation}
and decompose $T_{1,n}$ in two contributions:
\begin{equation*}
  \frac{\Delta t_F}{\mu} \underset{j=1}{\overset{n}{\sum}} \norm{\overline{\pt} u_h^j -  u_t(t^j)}_{L^2(\Omega)}  \leq \frac{\Delta t_F}{\mu}  \underset{j=1}{\overset{n}{\sum}} \big(   \norm{\overline{\pt} \theta^j_u}_{L^2(\Omega)}+\norm{w^j_{u}}_{L^2(\Omega)}  \big),
\end{equation*}
where
\begin{equation}
   \label{w1etw2u}
  w_{u}^j:=w_{1,u}^j+w_{2,u}^j \textrm{ with }  w_{1,u}^j:=(P_h^1-I)\overline{\pt} u(t^j), \textrm{\quad and \quad } w_{2,u}^j:=\overline{\pt} u(t^j) - u_t(t^j).
\end{equation}

Then by using the Cauchy-Schwarz inequality, 
\begin{equation}
 \label{decompositiondisc}
 T_{1,n}:= \frac{\Delta t_F}{\mu} \underset{j=1}{\overset{n}{\sum}} \norm{\overline{\pt} u_h^j - u_t(t^j)}_{L^2(\Omega)} \leq \frac{\sqrt{t^n-t^0}}{\mu} \big[\big(\underbrace{ \underset{j=1}{\overset{n}{\sum}} \Delta t_F \norm{\overline{\pt} \theta_u^j }^2_{L^2(\Omega)}}_{T_{\theta}} \big)^{1/2}+\big( \underbrace{\underset{j=1}{\overset{n}{\sum}} \Delta t_F \norm{ w_u^j }^2_{L^2(\Omega)}}_{T_w} \big)^{1/2} \big].
\end{equation}
\begin{itemize}
    \item Let us begin by the estimate on $T_{\theta}$.
On the state solution $u$, by choosing $v=\overline{\pt} \theta^n_u$, from \eqref{varpara2disc} (the operator $\overline{\pt}$ and the spatial derivative commute), we have
\begin{equation}
   \label{eq:ptthetaeq0}
  \norm{\overline{\pt} \theta^n_u}_{L^2(\Omega)}^2 + \mu (\nabla \theta_u^n,\overline{\pt} \nabla \theta^n_u)  = -(w^n_u,\overline{\pt} \theta^n_u),
\end{equation}
where $\theta_u^n$ is given by $\eqref{errorParabolU}$.
By definition of $\overline{\pt}$ (and with Young's inequality), 
  \begin{equation*}
  \norm{\overline{\pt} \theta^n_u}_{L^2(\Omega)}^2 + \frac{\mu}{\Delta t_F} \norm{\nabla \theta^n_u}_{L^2(\Omega)}^2 \leq \frac{\mu}{2 \Delta t_F} \big(\norm{\nabla \theta_u^n}_{L^2(\Omega)}^2 +\norm{ \nabla \theta_u^{n-1}}_{L^2(\Omega)}^2\big) + \frac{1}{2} \big( \norm{w^n_u}_{L^2(\Omega)}^2 + \norm{\overline{\pt} \theta^n}_{L^2(\Omega)}^2\big),
\end{equation*}
which entails
\begin{equation}
  \label{enfaitutile}
  \norm{\overline{\pt} \theta^n_u}_{L^2(\Omega)}^2 \leq \frac{\mu}{\Delta t_F} \norm{\nabla \theta_u^{n-1}}_{L^2(\Omega)}^2 -\frac{\mu}{\Delta t_F} \norm{\nabla \theta_u^n}_{L^2(\Omega)}^2  + \norm{w^n_u}_{L^2(\Omega)}^2,\ \forall n=1,\dots, N_T.
\end{equation}
Summing over the time steps, we get
\begin{equation*}
  \underset{j=1}{\overset{n}{\sum}}  \norm{\overline{\pt} \theta^j_u}_{L^2(\Omega)}^2 \leq \big| \underset{j=1}{\overset{n}{\sum}} \frac{\mu}{\Delta t_F}\big[\norm{\nabla \theta_u^{j-1}}_{L^2(\Omega)}^2 - \norm{\nabla \theta_u^j}_{L^2(\Omega)}^2 \big] + \norm{w^j_u}_{L^2(\Omega)}^2 \big|,
\end{equation*}
and we obtain from the initial condition $\theta_u^0=0$, 
\begin{equation*}
  \underset{j=1}{\overset{n}{\sum}}  \norm{\overline{\pt} \theta^j_u}_{L^2(\Omega)}^2 \leq \big| \frac{\mu}{\Delta t_F} \big( \norm{\nabla \theta_u^0}_{L^2(\Omega)}^2 - \norm{\nabla \theta_u^n}_{L^2(\Omega)}^2 \big) \big| +  \underset{j=1}{\overset{n}{\sum}} \norm{w^j_u}_{L^2(\Omega)}^2 \leq \frac{\mu}{\Delta t_F} \norm{\nabla \theta_u^n}_{L^2(\Omega)}^2 +  \underset{j=1}{\overset{n}{\sum}} \norm{w^j_u}_{L^2(\Omega)}^2.
\end{equation*}

From \eqref{enfaitutile} and by repeated application, we find for the first right-hand side term that
\begin{equation*}
   \norm{\nabla \theta_u^n}_{L^2(\Omega)}^2  \leq \frac{\Delta t_F}{\mu} \underset{j=1}{\overset{n}{\sum}}  \norm{w^j_u}_{L^2(\Omega)}^2,
\end{equation*}
and thus, multiplying by $\Delta t_F$ we find that
\begin{equation}
  \label{utilepourh1}
  T_{\theta}:=\underset{j=1}{\overset{n}{\sum}} \Delta t_F \norm{\overline{\pt} \theta^j_u}_{L^2(\Omega)}^2  \leq 2 \underset{j=1}{\overset{n}{\sum}} \Delta t_F  \norm{w^j_u}_{L^2(\Omega)}^2.
\end{equation}
Now, going back to \eqref{decompositiondisc}, we obtain
\begin{equation}
   \label{newdecompositiondisc}
  T_{1,n} \lesssim \frac{1}{\mu} \big( \underbrace{ \underset{j=1}{\overset{n}{\sum}}  \Delta t_F \norm{w^j_{u}}_{L^2(\Omega)}^2  }_{T_w}\big)^{1/2}  \lesssim \frac{1}{\mu} \big(  \underset{j=1}{\overset{n}{\sum}}  \Delta t_F \big[\norm{w^j_{1,u}}_{L^2(\Omega)}^2 + \norm{w^j_{2,u}}_{L^2(\Omega)}^2\big] \big)^{1/2}.
\end{equation}
\item It remains to estimate $T_w$.

\begin{itemize}
\item Let us first consider the construction for $w_{1,u}$
  \begin{equation*}
w_{1,u}^j=(P_h^1-I)\overline{\pt}u(t^j) =\frac{1}{\Delta t_F}(P_h^1-I)\int_{t^{j-1}}^{t^j}u_t \ \textrm{ ds }=\frac{1}{\Delta t_F}\int_{t^{j-1}}^{t^j} (P_h^1-I) u_t \ \textrm{ ds}, 
\end{equation*}
since $P_h^1$ and the time integral commute.
  Thus, from the Cauchy-Schwarz inequality,
  \begin{align}
    \label{w1L2}
  \Delta t_F \sum_{j=1}^n \norm{ w^j_{1,u}}_{L^2(\Omega)}^2 &\leq \Delta t_F \sum_{j=1}^n \int_{\Omega} \big[\frac{1}{\Delta t_F^2}\ \int_{t^{j-1}}^{t^j} ((P_h^1-I) u_t)^2 \ \textrm{ ds} \big] \ \Delta t_F  \nonumber\\
  &\leq \sum_{j=1}^n \int_{t^{j-1}}^{t^j} \norm{(P_h^1-I) u_t}^2_{L^2(\Omega)} \ \textrm{ ds}, \nonumber \\
&\lesssim Ch^4 \sum_{j=1}^n \int_{t^{j-1}}^{t^j} \norm{u_t}_{H^2(\Omega)}^2 \leq  h^4 \int_0^{t^n} \norm{u_t}_{H^2(\Omega)}^2 \ \textrm{ ds}. 
\end{align}
\item To estimate the $L^2$ norm of $w_{2,u}$, we write
\begin{equation*}
w_{2,u}^j=\frac{1}{\Delta t_F} (u(t^j)-u(t^{j-1}))-u_t(t^j)=-\frac{1}{\Delta t_F}\int_{t^{j-1}}^{t^j} (s-t^{j-1})u_{tt}(s) \ \textrm{ ds},
\end{equation*}
such that we end up with 
\begin{equation}
  \label{w2L2}
  \Delta t_F \sum_{j=1}^n \norm{ w_{2,u}^j}_{L^2(\Omega)}^2 \leq \sum_{j=1}^n\norm{ \int_{t^{j-1}}^{t^j} (s-t^{j-1})  u_{tt}(s) \ \mathrm{ds}}^2_{L^2(\Omega)} \leq  \Delta t_F^2 \int_0^{t^n} \norm{  u_{tt}}_{L^2(\Omega)}^2 \ \mathrm{ds}.
\end{equation}
\end{itemize}
\end{itemize}
\item Then, the bound of $T_{2,n}$, which appears in \eqref{thirdexpression}, is classical and can be found in \cite{thomee2}.
We have 
\begin{equation}
  \label{combinesend}
 T_{2,n} \lesssim h^2  \int_0^{t^n} \norm{\Psi_t}_{H^2(\Omega)} \textrm{ ds } + \Delta t_F \int_0^{t^n} \norm{\Psi_{tt}}_{L^2(\Omega)}\textrm{ ds },
  \end{equation}
  \end{itemize}
and the proof ends by using \eqref{rhoclassiquedisc}, \eqref{thirdexpression}, \eqref{newdecompositiondisc}, \eqref{w1L2}, \eqref{w2L2}, and \eqref{combinesend}.
\end{itemize}
\end{proof}


\section*{Acknowledgment}
This work is supported by the SPP2311 program. We would like to give special thanks to Ole Burghardt for his precious help on Automatic Differentiation.
\bibliography{SensitivityNIRB2}
\bibliographystyle{plain}
\end{document}